\documentclass[12pt]{article}

\usepackage{amsmath,amsthm,amsfonts,amssymb, color,graphicx}

\newtheorem{theorem}{Theorem}[section]

\newtheorem{proposition}[theorem]{Proposition}
\newtheorem{lemma}[theorem]{Lemma}
\newtheorem{definition}[theorem]{Definition}
\newtheorem{remark}[theorem]{Remark}

\def\cD{\mathcal{D}}

\def\cF{\mathcal{F}}
\def\cG{\mathcal{G}}
\def\cI{\mathcal{I}}
\def\cJ{\mathcal{J}}
\def\cH{\mathcal{H}}
\def\cR{\mathcal{R}}

\def\cX{\mathcal{X}}
\def\cY{\mathcal{Y}}
\def\cZ{\mathcal{Z}}

\def\bE{\mathbb{E}}
\def\bR{\mathbb{R}}
\def\bZ{\mathbb{Z}}

\def\e{\varepsilon}

\begin{document}

\title{Stratonovich solution for the wave equation} 

\author{Raluca M. Balan\footnote{University of Ottawa, Department of Mathematics and Statistics, 150 Louis Pasteur Private, Ottawa, Ontario, K1G 0P8, Canada. E-mail address: rbalan@uottawa.ca.}\footnote{Research supported by a grant from the Natural Sciences and Engineering Research Council of Canada.}}

\date{May 15, 2021}
\maketitle

\begin{abstract}
\noindent In this article, we construct a Stratonovich solution for the stochastic wave equation in spatial dimension $d \leq 2$, with time-independent noise and linear term $\sigma(u)=u$ multiplying the noise. The noise is spatially homogeneous and its spectral measure satisfies an integrability condition which is stronger than Dalang's condition. We give a probabilistic representation for this solution, similar to the Feynman-Kac-type formula given in \cite{DMT08} for the solution of the stochastic wave equation with spatially homogeneous Gaussian noise, that is white in time. We also give the chaos expansion of the Stratonovich solution and we compare it with the chaos expansion of the Skorohod solution from \cite{BCC}.
\end{abstract}


\vspace{1mm}


\tableofcontents

\section{Introduction}

In this article, we study the stochastic wave equation with time-independent noise:
\begin{align}
\label{wave} 
\begin{cases}
\dfrac{\partial^2 v}{\partial t^2} (t,x)=\Delta v(t,x)+v(t,x)\dot{W}(x) \quad \quad t>0,x \in \bR^d \quad (d \leq 2) \\
v(0,x) = 1, \quad \dfrac{\partial v}{\partial t}(0,x)=0 \quad \quad \quad \quad \quad \ x \in \bR^d
\end{cases}
\end{align}
The noise is given by a zero-mean Gaussian process $\{W(\varphi);\varphi \in \cD(\bR^d)\}$ defined on a complete probability space $(\Omega,\cF,P)$, with covariance
\[
\bE[W(\varphi)W(\psi)]=\int_{\bR^d} \int_{\bR^d}\varphi(x)\psi(y)\gamma(x-y)dxdy:=\langle \varphi,\psi\rangle_{\cH},
\]
where $\cD(\bR^d)$ is the space of infinitely differentiable functions on $\bR^d$ with compact support and $\gamma:\bR^d \to [0,\infty]$ is a non-negative-definite function. By the Bochner-Schwartz theorem, there exists a tempered measure $\bR^d$ such that $\gamma=\cF \mu$ in the sense of distributions, and hence
\[
\langle \varphi,\psi\rangle_{\cH}=
\int_{\bR^d}\cF \varphi(\xi)\overline{\cF \psi(\xi)}\mu(d\xi) \quad \mbox{for any} \quad \varphi,\psi \in \cD(\bR^d).
\]

Let $\cH$ be the completion of $\cD(\bR^d)$ with respect to $\langle \cdot, \cdot \rangle_{\cH}$. Then $W=\{W(\varphi);\varphi \in \cH\}$ is an isonormal Gaussian process, and we can use Malliavin calculus with respect to $W$ to analyze the solution of \eqref{wave}.
We recall briefly some basic facts from Malliavin calculus which are needed in this paper. We refer the reader to \cite{nualart06} for more details.
 For any $n\geq 1$, we denote by $\cH_n$ the $n$-th Wiener chaos space corresponding to $W$, i.e. the closed linear subspace of $L^2(\Omega)$ generated by $\{H_n(W(\varphi)); \varphi \in \cH, \|\varphi\|_{\cH}=1\}$, where $H_n(x)$ is the Hermite polynomial of order $n$. 
Any random variable $F \in L^2(\Omega)$, which is measurable with respect to $W$, has the Wiener chaos expansion:
\[
F=E(F)+\sum_{n\geq 1}I_n(f_n)
\]
where $I_n: \cH^{\otimes n} \to \cH_n$ is the multiple Wiener integral of order $n$. 
We let $\cH_0=\bR$ and $I_0:\bR \to \bR$ be the identity map. 
Moreover,
\[
\bE|F|^2=\big(\bE (F)\big)^2+ \sum_{n\geq 1}n! \| \widetilde{f}_n\|_{\cH^{\otimes n}}^2,
\]
where $\widetilde{f}_n$ is the symmetrization of $f_n$. Moreover, 
$\|\widetilde{f}\|_{\cH^{\otimes n}} \leq \|f\|_{\cH^{\otimes n}}$ for any $f \in \cH^{\otimes n}$.

\medskip
 
We impose the following assumption, which is needed for Lemmas \ref{energy-lemma} and \ref{lemmaF} below:

\medskip

{\bf Assumption A.}
(a) $\gamma$ is continuous, symmetric, and $\gamma(x)<\infty$ if and only if $x\not=0$; \\
(b) $\mu$ is absolutely continuous with respect to the Lebesgue measure.

\medskip

A basic example is the Riesz kernel $\gamma(x)=|x|^{-\alpha}$ with $\alpha \in (0,d)$, in which case $\mu(d\xi)=C_{d,\alpha}|\xi|^{-(d-\alpha)}d\xi$ and $C_{d,\alpha}>0$ is a constant depending on $d$ and $\alpha$.

\medskip

We are interested in the Stratonovich solution of \eqref{wave}, whose definition is given below.

\begin{definition}
{\rm
The random field $\{v(t,x);t\geq 0,x\in \bR^d\}$ is a (mild) {\bf Stratonovich solution} of equation \eqref{wave} if for any $t>0$ and $x \in \bR^d$, with probability 1,
\[
v(t,x)=1+\int_0^t \left(\int_{\bR^d}G(t-s,x-y)v(s,y)W^{\circ}(dy)\right) ds,
\]
where the $W^{\circ}(dy)$ denotes the Stratonovich integral.}
\end{definition}

We recall that the Stratonovich integral $\int_{\bR^d}\Phi(x)W^{\circ} (dx)$ of the random field $\{\Phi(x)\}_{x \in \bR^d}$ is defined as the following limit in probability, if it exists:
\[
\int_{\bR^d}\Phi(x)W^{\circ}(dx)=\lim_{\e \downarrow 0} \int_{\bR^d} \Phi(x)\dot{W}^{\e}(x)dx,
\]
where $\dot{W}^{\e}(x)=W(p_{\e}(x-\cdot))$ is a mollification of $W$ and $p_{\e}(x)=(2\pi \e)^{-d/2}e^{-|x^2|/(2\e)}$.

\medskip

We denote by $G$ the fundamental solution of the wave equation on $\bR_{+}\times \bR^d$:
\begin{align}
	G(t,x)=
	\begin{cases}
		\displaystyle \frac{1}{2}1_{\{|x|<t\}}                               & \text{if $d=1$},\\[1em]
		\displaystyle \frac{1}{2\pi} \frac{1}{\sqrt{t^2-|x|^2}}1_{\{|x|<t\}} & \text{if $d=2$},\\
	\end{cases}
\end{align}
Note that $G(t,\cdot)$ is integrable and $\int_{\bR^d}G(t,x)dx=t$.
We denote by $\cF \varphi(\xi)=\int_{\bR^d}e^{-i \xi \cdot x}\varphi(x)dx$
the Fourier transform of a function $\varphi \in L^1(\bR^d)$. Then
\[
\cF G(t,\cdot)(\xi)=\frac{\sin(t|\xi|)}{|\xi|} \quad \mbox{for all} \ \xi \in \bR^d.
\]

The following inequality will play an important role in the present article:
\begin{equation}
\label{FG-bound}
|\cF G(t,\cdot)(\xi)| \leq C_t \left( \frac{1}{1+|\xi|^2}\right)^{1/2},
\end{equation}
with $C_t=2\sqrt{2}(t \vee 1)$. To see this, use $\frac{\sin(t|\xi|)}{|\xi|} \leq 2(t \vee 1)\frac{1}{1+|\xi|}$ and $\left(\frac{1}{1+|\xi|}\right)^2 \leq \frac{2}{1+|\xi|^2}$.

\medskip

The first goal of this article is to prove the existence of a Stratonovich solution to equation \eqref{wave}.
This will be achieved under the following condition:
\begin{align}
\label{C-cond}
K_{\mu}:=\int_{\bR^d}\left( \frac{1}{1+|\xi|^2}\right)^{1/2} 
\mu(d\xi)<\infty,
\tag{C}
\end{align}
which, due to \eqref{FG-bound}, implies that
\begin{equation}
\label{sup-ineq}
\int_{\bR^d}|\cF G(t,\cdot)(\xi)|\mu(d\xi) \leq C_{t} K_{\mu}.
\end{equation}


The Stratonovich solution $v$ is different than the Skorohod solution $u$ of equation \eqref{wave}, which satisfies (by definition) the following integral equation:
\begin{equation}
\label{Skorohod}
u(t,x)=1+\int_0^t \int_{\bR^d} G(t-s,x-y)u(s,y)W(\delta y)ds,
\end{equation}
where $W(\delta y)$ denotes the Skorohod integral (see Section \ref{section-comparison} below).
In \cite{BCC}, it was shown that under {\em Dalang's condition}:
\begin{align}
\label{D-cond}
\int_{\bR^d}\frac{1}{1+|\xi|^2}\mu(d\xi)<\infty,
\tag{D}
\end{align}
the Skorohod solution exists and has the chaos expansion $u(t,x)=1+\sum_{n\geq 1}I_n\big( f_n(\cdot,x;t)\big)$, where
\begin{equation}
\label{def-fn}
f_n(x_1,\ldots,x_n,x;t)=\int_{T_n(t)} \prod_{i=1}^{n}G(t_{i+1}-t_i,x_{i+1}-x_i) dt_1 \ldots dt_n,
\end{equation}
and $T_n(t)=\{(t_1,\ldots,t_n) \in [0,t]^n;t_1<\ldots<t_n\}$. Here
$t_{n+1}=t$ and $x_{n+1}=x$. 

\medskip

If $\gamma$ satisfies the scaling property 
$\gamma(cx)=c^{-\alpha}\gamma(x)$ for all $c>0$ and $x\in \bR^d$,
for some $\alpha>0$, then $\mu(cA)=c^{\alpha}\mu(A)$ for any $c>0$ and $A \in {\cal B}(\bR^d)$, and (C) is equivalent to $\alpha\in (0,1)$
while (D) is equivalent to $\alpha \in (0,2)$.

\medskip
In the case of the heat equation:
\begin{equation}
\label{heat}
\begin{cases}
\dfrac{\partial u}{\partial t}(t,x)=\frac{1}{2}\Delta u(t,x)+u(t,x) \dot{W}(x), \quad t\geq 0,x\in \bR^d \ (d\geq 1) \\
\quad u(0,x)=1 \qquad \qquad \qquad \qquad \qquad \quad x\in \bR^d
\end{cases}
\end{equation}
it was proved in \cite{HHNT} that the Skorohod solution $u_h$ and the Stratonovich solution $v_h$ exist under Dalang's condition, and they admit the following Feynman-Kac representations:
\[
u_h(t,x)=\bE^{B}[e^{V(t,x)-\frac{1}{2}M(t)}] \quad \mbox{and} \quad v_{h}(t,x)=\bE^{B}[e^{V(t,x)}].
\]
Here $B=(B_t)_{t \geq 0}$ is a $d$-dimensional Brownian motion independent of $W$ and $B_t^x=B_t+x$. The functional $V(t,x)$, defined formally by
\[
V(t,x):=\int_0^t \int_{\bR^d}  \delta_0(B_r^x-y) W(dy)dr,
 \]
is the $L^2(\Omega)$-limit of $V^{\e}(t,x)= \int_0^t \int_{\bR^d}  p_{\e}(B_r^x-y) W(dy)dr$ as $\e \downarrow 0$,
 and the correction term
\[
M(t)=\int_0^t \int_0^t \gamma(B_s-B_r)dsdr
 \]
 is the conditional variance of $V(t,x)$ given $B$. If in addition,
\[
\int_{\bR^d}\left(\frac{1}{1+|\xi|^2} \right)^{1-\delta}\mu(d\xi)<\infty \quad \mbox{for some $\delta \in (0,1)$},
\]
then, by Proposition 5.28 of \cite{HHNT}, $W$ has a version with values in a weighted Besov space, and equation \eqref{heat} has a pathwise solution which coincides with the Stratonovich solution.
The construction of the Stratonovich solution of equation \eqref{heat} (given by Theorem 5.7 of \cite{HHNT}) relies heavily on the Feynman-Kac representation of this solution, and cannot be extended to the wave equation.

\medskip

In this article, we develop a method for constructing a Stratonovich solution based on chaos expansions, which seems to be new in the literature. We implement this method for the wave equation, but we believe that the method is so robust that can be applied to a large class of SPDEs with time-dependent noise and linear term $\sigma(u)=u$ multiplying the noise.
(We postpone the study of the wave equation with time-dependent noise for future work.)
We believe that this method can also be applied to show the existence of a Stratonovich solution for the heat equation 
under condition (D), using inequality:
\[
\int_0^T \cF G_h(t,\cdot)(\xi) dt \leq 4(T \vee 1) \frac{1}{1+|\xi|^2},
\]
instead of \eqref{FG-bound},
where $\cF G_{h}(t,\cdot)(\xi)=e^{-t|\xi|^2/2}$ is Fourier transform in the space variable of the fundamental solution $G_h(t,x)=p_t(x)$ of the heat equation on $\bR_{+}\times \bR^d$.

The idea is to use the classical method of Picard's iterations for solving the equation with mollified noise $W^{\e}$:
\begin{align}
\label{wave-m}
\begin{cases}
\dfrac{\partial^2 v^{\e}}{\partial t^2} (t,x)=\Delta v^{\e}(t,x)+v^{\e}(t,x)\dot{W}^{\e}(x), \quad \quad t>0,x \in \bR^d \quad  (d\leq 2)\\
v^{\e}(0,x) = 1, \quad \dfrac{\partial v^{\e}}{\partial t}(0,x)=0\end{cases}
\end{align}
and then apply the product formula from Malliavin calculus for writing the $n$-th Picard approximation $v_n^{\e}(t,x)$ of $v^{\e}(t,x)$ as a finite sum of multiple Wiener integrals. The delicate part is to let $\e \downarrow 0$. Finally, we let $n \to \infty$.
This method involves some non-trivial combinatorial arguments, which appeared for the first time in the earlier works \cite{hu-meyer1,hu-meyer2} of Hu and Meyer on multiple Stratonovich integrals with respect to Brownian motion.

\medskip

We are now ready to state the first main result of this paper.

\begin{theorem}
\label{main}
If Assumption A holds and $\mu$ satisfies condition (C), then there exists a process $v=\{v(t,x);t\geq 0,x \in \bR^d\}$ such that for any $T>0$ and $p\geq 2$,
\begin{equation}
\label{ve-v-conv}
\sup_{(t,x)\in [0,T] \times \bR^d}\bE|v^{\e}(t,x)-v(t,x)|^p \to 0 \quad \mbox{as $\e\downarrow 0$},
\end{equation}
where $v^{\e}$ is the solution of \eqref{wave-m}. The random variable $v(t,x)$ can be represented as
\begin{equation}
\label{series-v}
v(t,x)=1+\sum_{n\geq 1}H_n(t,x),
\end{equation}
where $H(t,x)=I_n^{\circ}\big(f_n(\cdot,x;t)\big)$ is the multiple Stratonovich integral of $f_n(\cdot,x;t)$, and is given by relation \eqref{def-Hn} below.
Moreover, $v$ is a Stratonovich solution of equation \eqref{wave}.
\end{theorem}

The series \eqref{series-v} can be used to derive the Wiener chaos expansion of $v(t,x)$, but the explicit expression of the projection on each Wiener chaos space is quite involved, since $H_n(t,x)$ is a sum of multiple Wiener integrals of order $n-2k$, for $k=0,1,\ldots,\lfloor n/2 \rfloor$. Note that the term of this sum which corresponds to $k=0$ is the $n$-th chaos term appearing in the Wiener chaos expansion of the Skorohod solution $u(t,x)$ (see Section \ref{section-comparison} below).

\bigskip



The second goal of this article is to give a Feynman-Kac-type formula for the Stratonovich solution of equation \eqref{wave}, similar to the one given by Dalang, Mueller and Tribe in \cite{DMT08} for It\^o-Skorohod solution of the wave equation with spatially-homogeneous Gaussian noise that is white in time. Unlike the Feynman-Kac formula for the solution of the heat equation (which uses the paths of Brownian motion), this formula is based on a process whose paths are obtained by linear interpolation between the jump times of a Poisson process.

More precisely, let $N=(N_t)_{t\geq 0}$ be a Poisson process of rate $1$, with jump times $(\tau_i)_{i\geq 1}$. Denote $\tau_0=0$. Let $(U_i)_{i\geq 1}$ be i.i.d. random variables with values in $\bR^d$, with density $G(1,\cdot)$. Note that if  $d=1$, $U_1$ has a uniform distribution on $(-1,1)$. Assume that $N$, $(U_i)_{i\geq 1}$ and $W$ are independent.

Based on $N$ and $(U_i)_{i\geq 1}$, we define the linearly interpolated process $(X_t)_{t\geq 0}$ as follows:
\[
X_t=X_{\tau_{i}}+(t-\tau_i)U_{i+1} \quad \mbox{for $\tau_i<t \leq \tau_{i+1}$, $i\geq 0$,}
\]
with $X_0=0$. Let $X_t^x=x+X_t$.

\medskip

The next theorem is the second main result of this paper, which gives the Feynman-Kac-type representation for $v$.

\begin{theorem}
\label{FK-th}
If Assumption A holds, $\mu$ satisfies condition (C), and $v$ is the Stratonovich solution of equation \eqref{wave} given by Theorem \ref{main}, then for any $t>0,x \in \bR^d$, almost surely,
\begin{align}
\label{FK}
v(t,x)&=e^t \bE^{N,X} \Big[ \prod_{i=1}^{N_t} (\tau_i-\tau_{i-1}) \prod_{i=1}^{N_t} \dot{W}(X_{\tau_i}^x)\Big] \\
\nonumber
&:=e^{t} \sum_{n\geq 0}
\bE^{N,X} \Big[1_{\{N_t=n\}} \prod_{i=1}^{N_t} (\tau_i-\tau_{i-1}) \prod_{i=1}^{N_t} \dot{W}(X_{\tau_i}^x)\Big],
\end{align}
where the terms of the series are defined by relation \eqref{def-FK-Hn} below if $n\geq 1$, and we use the convention $\prod_{i=1}^{0}=1$ if $n=0$ (so that the first term of the series is $e^t\bE^{N,X}[1_{\{N_t=0\}}]=1$). 
\end{theorem}

\bigskip

This article is organized as follows. In Section \ref{section-ve}, we analyze equation \eqref{wave-m} with mollified noise. The proofs of Theorems \ref{main} and \ref{FK-th} are given in Sections \ref{section-construction} and \ref{section-FK}, respectively.
In Section \ref{section-mu-finite}, we examine the particular case of a very smooth noise $W$, which has finite spectral measure $\mu$. In Section \ref{section-comparison}, we give the chaos expansion of the Stratonovich solution and compare it with the chaos expansion of the Skorohod solution. Appendices A-C contain some auxiliary results needed in the sequel.

\section{Equation with mollified noise}
\label{section-ve}

In this section, we study equation \eqref{wave-m} with mollified noise $\dot{W}^{\e}$.

\subsection{Existence of solution}

The mollified noise $\{\dot{W}^{\e}(x)\}_{x \in \bR^d}$ is a zero-mean Gaussian process with stationary increments (in Yaglom sense) and covariance function:
\[
\bE|\dot{W}^{\e}(x)-\dot{W}^{\e}(y)|^2=\int_{\bR^d}(1-e^{i\xi \cdot (x-y)})^2 e^{-\e|\xi|^2}\mu(d\xi).
\]

Under condition (D),
\[
\bE|\dot{W}^{\e}(x)|^2=\|p_{\e}(x-\cdot)\|_{\cH}^2=\int_{\bR^d}e^{-\e|\xi|^2}\mu(d\xi) \leq \int_{\bR^d}\frac{1}{1+\e |\xi|^2}\mu(d\xi)=:C_{\mu,\e}<\infty
\]
and
\begin{equation}
\label{bound-We}
\bE|\dot{W}^{\e}(x)|^p=z_{p} (\bE|\dot{W}^{\e}(x)|^2)^{p/2}\leq z_p C_{\mu,\e}^{p/2} \quad \mbox{for any $p>0$},
\end{equation}
where
\[
z_p=E|Z|^p=\frac{2^{p/2}}{\sqrt{\pi}}\Gamma\left(\frac{p+1}{2}\right) \quad \mbox{with $Z \sim N(0,1)$}.
\]

\begin{definition}
{\rm
We say that the random field $v^{\e}=\{v^{\e}(t,x); t\geq 0, x\in \bR^d\}$ is a {\bf solution} of equation \eqref{wave-m} if for any $t>0$ and $x \in \bR^d$, with probability $1$,
\begin{equation}
\label{wave-m-1}
v^{\e}(t,x)=1+\int_0^t \int_{\bR^d}G(t-s,x-y)v^{\e}(s,y)\dot{W}^{\e}(y)dyds.
\end{equation}
}
\end{definition}

Intuitively, $v^{\e}(t,x)$ should be given by the series:
\begin{equation}
\label{ve-series}
v^{\e}(t,x)=1+\sum_{n\geq 1}\int_{T_n(t)} \int_{\bR^{nd}} \prod_{i=1}^{n}G(t_{i+1}-t_{i},x_{i+1}-x_{i}) \prod_{i=1}^{n}\dot{W}^{\e}(x_i)d{\bf x} d{\bf t},
\end{equation}
where ${\bf x}=(x_1,\ldots,x_n)$, ${\bf t}=(t_1,\ldots,t_n)$ and we let $t_{n+1}=t$ and $x_{n+1}=x$.
 
\medskip

Note that \eqref{ve-series} cannot be obtained by a pathwise application of Proposition 2.2 of \cite{DMT08}, since the potential $\dot{W}^{\e}(x)$ is not bounded.
We will justify below relation \eqref{ve-series} for fixed $\e>0$. The fact that the series \eqref{ve-series} converges {\em uniformly} in $\e$ will be proved in Section \ref{section-n-conv}.

\medskip

Let $
v_n^{\e}(t,x)=1+\sum_{k=1}^{n}H_k^{\e}(t,x)$ be the $n$-th Picard iteration, where
\begin{equation}
\label{def-Hne}
H_n^{\e}(t,x)=\int_{T_n(t)} \int_{\bR^{nd}} \prod_{i=1}^{n}G(t_{i+1}-t_i,x_{i+1}-x_i) \prod_{i=1}^{n}\dot{W}^{\e}(x_i)d{\bf x} d{\bf t}.
\end{equation}

Note that $H_n^{\e}(t,x)$ is well-defined due to
inequality \eqref{bound-Hne} below.
Let $H_0^{\e}(t,x)=1$. We have the following recurrence relation: for any $n\geq 0$,
\[
H_{n+1}^{\e}(t,x)=\int_{0}^{t}\int_{\bR^d}G(t-s,x-y)H_n^{\e}(s,y)\dot{W}^{\e}(y)dyds.
\]

\medskip

The next result shows that the series \eqref{ve-series} converges, for any $\e>0$ fixed.

\begin{lemma}
\label{vne-conv}
Under condition (D), for any $\e>0$ fixed, $\lim_{n \to \infty}v_n^{\e}(t,x)=v^{\e}(t,x)$ exists in $L^p(\Omega)$, uniformly in $[0,T] \times \bR^d$, for any $p\geq 1$ and $T>0$, i.e.
 \[
\sup_{(t,x) \in [0,T] \times \bR^d}\bE|v_n^{\e}(t,x)-v^{\e}(t,x)|^p \to 0 \quad \mbox{as $n \to \infty$}.
 \]
 Moreover, $v^{\e}$ is a solution of equation \eqref{wave-m}.
\end{lemma}

\begin{proof} Let $\|\cdot\|_p$ be the norm in $L^p(\Omega)$.
Using Minkowski inequality, the generalized H\"older inequality and inequality \eqref{bound-We}, we obtain that for any $p \geq 1$,
\begin{align*}
\|H_n^{\e}(t,x)\|_p & \leq \int_{T_n(t)} \int_{\bR^{nd}} \prod_{i=1}^{n}G(t_{i+1}-t_i,x_{i+1}-x_i)\Big\|\prod_{i=1}^{n}\dot{W}^{\e}(x_i)\Big\|_p d{\bf x} d{\bf t} \\
 & \leq \int_{T_n(t)} \int_{\bR^{nd}} \prod_{i=1}^{n}G(t_{i+1}-t_i,x_{i+1}-x_i)\prod_{i=1}^{n}\|\dot{W}^{\e}(x_i)\|_{pn} d{\bf x} d{\bf t}\\
 &\leq z_{pn}^{1/p} C_{\e}^{n/2}\int_{T_n(t)} \int_{\bR^{nd}}\prod_{i=1}^{n}G(t_{i+1}-t_i,x_{i+1}-x_i)d{\bf x} d{\bf t} \\
 &=z_{pn}^{1/p} C_{\mu,\e}^{n/2} \int_{T_n(t)} \prod_{i=1}^{n}(t_{i+1}-t_i)d{\bf t}\\
 & =\left\{\frac{2^{pn/2}}{\sqrt{\pi}} \Gamma\left(\frac{pn+1}{2}\right)\right\}^{1/p} C_{\mu,\e}^{n/2} \frac{t^{2n}}{(2n)!}.
\end{align*}
 Since $\Gamma(\frac{pn+1}{2}) \leq C_p^n (n!)^{p/2}$ and $(2n)! \geq c_0^n (n!)^2$ (due to Stirling's formula), we obtain:
\begin{equation}
\label{bound-Hne}
\|H_n^{\e}(t,x)\|_p \leq \frac{2^{n/2}}{\pi^{1/(2p)}}C_{p}^{n/p} C_{\mu,\e}^{n/2}c_0^n \frac{t^{2n}}{(n!)^{3/2}}.
\end{equation}
Hence,
\[
\sum_{n\geq 0}\sup_{(t,x) \in [0,T] \times \bR^d}\|H_n^{\e}(t,x)\|_{p}<\infty
\]
and the sequence $\{v_n^{\e}(t,x)\}_{n}$ is Cauchy in $L^p(\Omega)$, uniformly in $(t,x) \in [0,T] \times \bR^d$.

To prove that $v^{\e}$ is a solution to \eqref{wave-m}, we let $n\to \infty$ in the recurrence relation:
\begin{equation}
\label{rec-vne}
v_{n+1}^{\e}(t,x)=1+\int_0^t \int_{\bR^d} G(t-s,x-y)v_n^{\e}(s,y)\dot{W}^{\e}(y)dyds.
\end{equation}
The left-hand side converges in $L^p(\Omega)$ to $v^{\e}(t,x)$. The right-hand side converges in $L^1(\Omega)$ to $\int_0^t \int_{\bR^d} G(t-s,x-y)v^{\e}(s,y)\dot{W}^{\e}(y)dyds$, by the dominated convergence theorem. To justify the application of this theorem, we use the fact that
\[
\bE|\big(v_n^{\e}(s,y)-v^{\e}(s,y)\big)\dot{W}^{\e}(y)| \leq \left(\bE|v_n^{\e}(s,y)-v^{\e}(s,y)|^2\right)^{1/2} \left(\bE|\dot{W}^{\e}(y)|^2\right)^{1/2}.
\]
The last term is bounded by a constant (that depends on $\e$), due to \eqref{bound-We} and the uniform convergence of $(v_{n}^{\e})_n$.
\end{proof}

\subsection{Feynman-Kac-type formula}

 In this section, we give a Feynman-Kac-type representation for $v^{\e}$. Note that this cannot be obtained directly by a pathwise application of Theorem 3.2 of \cite{DMT08} since the random potential $W^{\e}$ is not bounded.

 \medskip

 Recall the definition of the linearly interpolated process $(X_t)_{t\geq 0}$ given in the introduction.
Note that $X_{\tau_{i+1}}-X_{\tau_i}=(\tau_{i+1}-\tau_i)U_i$. Hence
$(X_{\tau_{i+1}}-X_{\tau_i})_{i\geq 0}$ are independent given $N$.
Since $tU_i$ has density $G(t,\cdot)/t$ for any $t>0$, $X_{\tau_{i+1}}-X_{\tau_i}$ has density $G(\tau_{i+1}-\tau_i,\cdot)/(\tau_{i+1}-\tau_i)$ given $N$. It follows that:
\begin{equation}
\label{cond-den}
(x_1,\ldots,x_n) \to
\prod_{i=1}^{n}\frac{G(\tau_i-\tau_{i-1},x_i-x_{i-1})}{\tau_i-\tau_{i-1}} \ \mbox{is the density of $(X_{\tau_1}^x,\ldots,X_{\tau_n}^x)$ given $N$},
\end{equation}
with the convention $x_0=x$.

\medskip

The key idea is the following identity, which gives a representation of an integral over the simplex $T_n(t)$ using the points of the Poisson process $N$: for any integrable function $h$ on $T_n(t)$,
\begin{equation}
\label{link-T-N}
\int_{T_n(t)}h(t_1,\ldots,t_n)dt_1 \ldots dt_n=e^t \bE[h(\tau_1,\ldots,\tau_n)1_{\{N_t=n\}}],
\end{equation}
 This identity follows from the fact that $(\tau_1,\ldots,\tau_n)$ has a uniform distribution over $T_n(t)$, given $N_t=n$.

\medskip

We will use the following general facts. If $X:\Omega \to \cX$ and $Y:\Omega \to \cY$ are independent random variables, then for any non-negative measurable functions $f,g,h$, we have
\begin{equation}
\label{e-id1}
\bE^X[f(X,Y)]=\bE[f(X,Y)|Y]
\end{equation}
\begin{equation}
\label{e-id2}
\bE^X[g(X)h(Y)]= \bE[g(X)] \, h(Y),
\end{equation}
where $\bE^X$ denotes the expectation with respect to $X$. More precisely, $\bE^X[f(X,Y)]$ is the random variable given by $\bE^X[f(X,Y)]=\int_{\cX}f(x,Y)\mu_{X}(dx)$, where $\mu_{X}$ is the law of $X$.

\medskip
 The next result gives the Feynman-Kac-type representation of $v^{\e}$.

\begin{lemma}
\label{lemma-FK-ve}
For any $\e>0$, $t>0$ and $x \in \bR^d$, with probability $1$,
\begin{equation}
\label{FK-ve}
v^{\e}(t,x)=e^t \bE^{N,X} \Big[ \prod_{i=1}^{N_t}(\tau_i-\tau_{i-1}) \prod_{i=1}^{N_t} \dot{W}^{\e}(X_{\tau_i}^x)\Big].
\end{equation}
\end{lemma}

\begin{proof} Fix $\e>0$, $t>0$ and $x \in \bR^d$. By Lemma \ref{vne-conv}, $v_n^{\e}(t,x) \to v^{\e}(t,x)$ a.s. along a subsequence. So it is enough to prove that for any $n\geq 1$, with probability $1$,
\[
v_n^{\e}(t,x)=e^{t}\bE^{N,X}\Big[  1_{\{N_t\leq n\}} \prod_{i=1}^{N_t}(\tau_i-\tau_{i-1}) \prod_{i=1}^{N_t}\dot{W}^{\e}(X_{\tau_i}^x)\Big],
\]
and then let $n\to \infty$. For this, it suffices to prove that for any $n \geq 1$, with probability $1$,
\begin{equation}
\label{FK-Hn-e}
H_n^{\e}(t,x)=e^{t}\bE^{N,X}\Big[  1_{\{N_t=n\}}\prod_{i=1}^{n}(\tau_i-\tau_{i-1}) \prod_{i=1}^{n}\dot{W}^{\e}(X_{\tau_i}^x) \Big].
\end{equation}
In definition \eqref{def-Hne} of $H_n(t,x)$, we use the fact that $G(t,x)=G(t,-x)$ and the change of variables
$\bar{t}_i=t-t_{n+1-i}$ and $\bar{x}_i=x_{n+1-i}$ for all $i=1,\ldots,n$. Then
\begin{equation}
\label{Hne-2}
H_n^{\e}(t,x)=\int_{T_n(t)} \int_{(\bR^{d})^n} \prod_{i=1}^{n}G(t_{i}-t_{i-1},x_{i}-x_{i-1}) \prod_{i=1}^{n}\dot{W}^{\e}(x_i) 
d{\bf x} d{\bf t},
\end{equation}
with $t_0=0$ and $x_0=x$.
We interchange the integrals. Using representation \eqref{link-T-N} for the integral over the simplex, followed by \eqref{e-id2} (since $W$ and $N$ are independent), we obtain:
\begin{align*}
H_n^{\e}(t,x)&=e^t \int_{(\bR^{d})^n} \bE\Big[ \prod_{i=1}^{n} G(\tau_i-\tau_{i-1},x_i-x_{i-1})1_{\{N_t=n\}}\Big] \prod_{i=1}^{n}\dot{W}^{\e}(x_i)d{\bf x}\\
&=e^t \int_{(\bR^{d})^n} \bE^{N}\Big[ \prod_{i=1}^{n} G(\tau_i-\tau_{i-1},x_i-x_{i-1})1_{\{N_t=n\}} \prod_{i=1}^{n}\dot{W}^{\e}(x_i)\Big]d{\bf x} \\
&=e^t \bE^N\Big[1_{\{N_t=n\}}\prod_{i=1}^{n}(\tau_i-\tau_{i-1}) \int_{(\bR^{d})^n}  \prod_{i=1}^{n} \frac{G(\tau_i-\tau_{i-1},x_i-x_{i-1})}{\tau_i-\tau_{i-1}} \prod_{i=1}^{n}\dot{W}^{\e}(x_i)d{\bf x}\Big].
\end{align*}
Relation \eqref{FK-Hn-e} follows using Lemma \ref{elem-lem} below with $X=(X_{\tau_1}^x,\ldots,X_{\tau_n}^x)$, $Y=N$, $Z=W$,
\[
g(x_1,\ldots,x_n,W)=\prod_{i=1}^{n}\dot{W}^{\e}(x_i) \quad \mbox{and} \quad
h(N)=1_{\{N_t=n\}} \prod_{i=1}^{n}(\tau_i-\tau_{i-1}),
\]
noticing that
\[
f_{(X_{\tau_1}^x,\ldots,X_{\tau_n}^x)|N}(x_1,\ldots,x_n|N)=
\prod_{i=1}^{n} \frac{G(\tau_i-\tau_{i-1},x_i-x_{i-1})}{\tau_i-\tau_{i-1}}
\]
(with $x_0=x$) is the conditional density of $(X_{\tau_1}^x,\ldots,X_{\tau_n}^x)$ given $N$ (see \eqref{cond-den}).
\end{proof}

\begin{lemma}
\label{elem-lem}
If random variables $X:\Omega\to \cX$, $Y:\Omega \to \cY$ and $Z:\Omega \to \cZ$ are such that $Z$ is independent of $(X,Y)$ and the conditional density $f_{X|Y}(x|Y)$ of $X$ given $Y$ exists, then for any non-negative measurable functions $g$ and $h$,
\[
\bE^Y\Big[h(Y)\int_{\cX} f_{X|Y}(x|Y)g(x,Z)dx \Big]=\bE^{Y,X} \Big[h(Y) g(X,Z) \Big].
\]
\end{lemma}

\begin{proof} We denote by $\mu_Y$ the law of $Y$ and $\mu_{Y,X}$ the law of $(Y,X)$. Then
\[
\mu_{Y,X}(B \times A)=\int_{B}P(X \in A|Y=y)\mu_{Y}(dy)=\int_{B}\int_{A}f_{X|Y}(x|y)dx\mu_{Y}(dy).
\]
Since $Y$ and $Z$ are independent, using \eqref{e-id1}, we have:
\[
\bE^Y\Big[h(Y)\int_{\bR^n} f_{X|Y}(x|Y)g(x,Z)dx \Big]=\bE\Big[h(Y)\int_{\bR^n} f_{X|Y}(x|Y)g(x,Z)dx\Big|Z\Big]=g_1(Z),
\]
where
\begin{align*}
g_1(z)&=\bE\Big[h(Y)\int_{\cX} f_{X|Y}(x|Y)g(x,z)dx \Big]=\int_{\cY} h(y)\int_{\bR^n} f_{X|Y}(x|y)g(x,z)dx \mu_{Y}(dy)\\
&=\int_{\cY \times \cX} h(y)g(x,z)\mu_{Y,X}(dy,dx).
\end{align*}
Since $(X,Y)$ and $Z$ are independent, using \eqref{e-id1}, we have:
\[
\bE^{Y,X} \Big[h(Y) g(X,Z) \Big]=\bE[\Big[h(Y) g(X,Z) |Z \Big]=g_2(Z),
\]
where $g_2(z)=\bE[h(Y) g(X,z)]$. The conclusion follows since $g_1(z)=g_2(z)$ for any $z$.
\end{proof}

\section{Proof of Theorem \ref{main}}
\label{section-construction}

In this section, we give the proof of Theorem \ref{main}, which contains also the construction of the Stratonovich solution $v$.

\medskip

As mentioned in the introduction, the proof is based on chaos expansions.
Recall that $v_n^{\e}$ is the $n$-th Picard approximations of $v^{\e}$.
The proof of \eqref{ve-v-conv} is summarized by the diagram below, in which all arrows indicate convergence in $L^{p}(\Omega)$, uniform in $(t,x)\in [0,T] \times \bR^d$:
\[
\begin{array}{ccccc}
& v_n^{\e}(t,x) & \longrightarrow_{\e} & v_n(t,x) & \mbox{for any $n$} \\
\mbox{(uniform in $\e$)} & \downarrow_n &  & \downarrow_n & \\
& v^{\e}(t,x) & --- \to_{\e} & v(t,x) &
\end{array}
\]
From this diagram, we see that there are are three facts that need to be proved:\\
{\em (i)} the convergence of $v_n^{\e}(t,x)$ to some $v_n(t,x)$ (to be determined), as $\e \downarrow 0$;\\
{\em (ii)} the convergence of $v_n(t,x)$ to some $v(t,x)$ (to be determined), as $n \to \infty$;\\
{\em (iii)} the convergence of $v_n^{\e}(t,x)$ to $v^{\e}(t,x)$ as $n \to \infty$, uniform in $\e$.

\medskip

In Section \ref{section-multiple-Str}, we examine briefly the multiple Stratonovich integrals. We will prove {\em (i)} in Section \ref{section-e-conv}, and
{\em (ii)}-{\em (iii)} in Section \ref{section-n-conv}.
The fact that $v$ is a Stratonovich solution will be proved in Section \ref{section-Stratonovich}.

\subsection{Multiple Stratonovich integrals}
\label{section-multiple-Str}

In this section, we examine the integral $I_n^{\circ}(f)$.
First, note
\begin{equation}
\label{pe*pe}
\langle p_{\e}(x_1-\cdot),p_{\e}(x_2-\cdot) \rangle_{\cH}=(p_{2\e}*\gamma)(x_1-x_2),
\end{equation}
and $p_{2\e}*\gamma=\cF \mu_{\e}$ in the sense of distributions, where
\begin{equation}
\label{def-mu-e}
\mu_{\e}(d\xi)=e^{-\e|\xi|^2} \mu(d\xi).
\end{equation}

We need to introduce some notation.
We let $[n]=\{1,\ldots,n\}$. For any set $J \subset [n]$, we let $|J|$ be the cardinality of $J$, and $J^c$ be the complement of $J$ in $[n]$. We denote $\lfloor x \rfloor=k$ if $k \in \bZ$ and $k \leq x<k+1$.

\begin{definition}
{\rm
The {\bf multiple Stratonovich integral} of order $n$ of a measurable function $f:\bR^d \to \bR$ is defined as the following limit in probability, if it exists:
\[
I_n^{\circ}(f):=\int_{(\bR^d)^n} f(x_1,\ldots,x_n) \prod_{i=1}^{n}W^{\circ}(dx_i)=\lim_{\e \downarrow 0} \int_{(\bR^d)^n} f(x_1,\ldots,x_n) \prod_{i=1}^{n}\dot{W}^{\e}(x_i) dx_1 \ldots dx_n.
\]
}
\end{definition}

Using Theorem \ref{product-n} for expressing the product $\prod_{i=1}^{n}\dot{W}^{\e}(x_i)$ and \eqref{pe*pe}, we obtain that:
\begin{align}
\label{W-product-n}
\prod_{i=1}^{n}\dot{W}^{\e}(x_i)&=\sum_{k=0}^{\lfloor n/2 \rfloor} \sum_{\substack{J \subset [n]\\ |J|=n-2k}} \sum_{\substack{\{I_1,\ldots,I_k\} \ {\rm partition} \ {\rm of} \ J^c \\ I_i=\{\ell_i,m_i\} \forall i=1,\ldots,k}} I_{n-2k} \big(\bigotimes_{j \in J}p_{\e}(x_j-\cdot) \big)\prod_{i=1}^{k}(p_{2\e}*\gamma)(x_{\ell_i}-x_{m_i}).
\end{align}

This shows that $I_n^{\circ}(f)$ exists if $I_{n-2k}(f_{k,J,I_1,\ldots,I_k}^{\e}) \stackrel{P}{\to} I_{n-2k}(f_{k,J,I_1,\ldots,I_k})$ as $\e \downarrow 0$, for any $k=0,1,\ldots,\lfloor n/2 \rfloor$, for any set $J \subset [n]$ and for any partition $\{I_1,\ldots,I_k\}$ as above,
where
\begin{align*}
f_{k,J,I_1,\ldots,I_k}^{\e}\big( (y_j)_{j\in J} \big)&=\int_{(\bR^d)^n} f(x_1,\ldots,x_n) \prod_{j \in J}p_{\e}(x_j-y_j) \prod_{i=1}^{k}(p_{2\e}*\gamma)(x_{\ell_i}-x_{m_i}) dx_1 \ldots dx_n\\
f_{k,J,I_1,\ldots,I_k}\big( (x_j)_{j\in J} \big)&=\int_{(\bR^d)^{2k}} f(x_1,\ldots,x_n) \prod_{i=1}^{k}\gamma(x_{\ell_i}-x_{m_i}) d \big( (x_j)_{j \in J^c}\big).
\end{align*}
Moreover, in this case,
\[
I_n^{\circ}(f)=\sum_{k=0}^{\lfloor n/2 \rfloor} \sum_{\substack{J \subset [n]\\ |J|=n-2k}} \sum_{\substack{\{I_1,\ldots,I_k\} \ {\rm partition} \ {\rm of} \ J^c \\ I_i=\{\ell_i,m_i\} \forall i=1,\ldots,k}} I_{n-2k}(f_{k,J,I_1,\ldots,I_k}).
\]

For $n=1$, $I_1^{\circ}(f)=\int_{\bR^d}f(x)W(dx)$ if $f \in \cH$, and hence $W^{\circ}(dx)=W(dx)$. For $n=2$, 
\begin{align*}
I_2^{\circ}(f)
&=\int_{(\bR^{d})^2} f(x_1,x_2)W(dx_1)W(dx_2)+\int_{(\bR^{d})^2} f(x_1,x_2)\gamma(x_1-x_2)dx_1 dx_2,
\end{align*}
provided that both integrals above are well-defined, and hence
\[
W^{\circ}(dx_1)W^{\circ}(dx_2)=W(dx_1)W(dx_2)+\gamma(x_1-x_2)dx_1 dx_2.
\]
For $n=3$, 
we have:
\begin{align*}
W^{\circ}(dx_1)W^{\circ}(dx_2)W^{\circ}(dx_3)&=W(dx_1)W(dx_2)W(dx_3)+W(dx_1)\gamma(x_2-x_3)dx_2 dx_3+\\
& \quad  W(dx_2)\gamma(x_1-x_3)dx_1 dx_3 +W(dx_3)\gamma(x_1-x_2)dx_1 dx_2.
\end{align*}
In general, 
\begin{align}
\label{def-product-V}
\prod_{i=1}^{n}W^{\circ}(dx_i)=\sum_{k=0}^{\lfloor n/2 \rfloor} \sum_{\substack{J \subset [n]\\ |J|=n-2k}} \sum_{\substack{\{I_1,\ldots,I_k\} \ {\rm partition} \ {\rm of} \ J^c \\ I_i=\{\ell_i,m_i\} \forall i=1,\ldots,k}}\prod_{j \in J}W(dx_j) \prod_{i=1}^{k}\gamma(x_{\ell_i}-x_{m_i})d\big((x_j)_{j \in J^c}\big).
\end{align}
Formally, we can say that
\begin{equation}
\label{formal-V}
\prod_{i=1}^n W^{\circ}(dx_i)=\lim_{\e \downarrow 0}\prod_{i=1}^{n}\dot{W}^{\e}(x_i) dx_1 \ldots dx_n =\prod_{i=1}^{n}\dot{W}(x_i) dx_1 \ldots dx_n.
\end{equation}

\subsection{Convergence as $\e \downarrow 0$}
\label{section-e-conv}

In this section, we identify a random variable $v_n(t,x)$ such that $v_n^{\e}(t,x)\to v_n(t,x)$ in $L^p(\Omega)$, as $\e \downarrow 0$, uniformly in $(t,x) \in [0,T] \times \bR^d$, for any $p\geq 2$.

\medskip

The next result shows that the multiple Stratonovich integral of the kernel $f_n(\cdot,x;t)$ exists.
 For any ${\bf t}=(t_1,\ldots,t_n)\in T_n(t)$ and ${\bf x}=(x_1,\ldots,x_n) \in (\bR^d)^n$, we define:
\begin{equation}
\label{def-h}
h_{{\bf t},(x_j)_{j\in J^c}}^{(t,x)}\big((x_j)_{j \in J}\big):=\prod_{j=1}^{n} G(t_{j+1}-t_j,x_{j+1}-x_j),
\end{equation}
with $t_{n+1}=t$ and $x_{n+1}=x$.

\begin{theorem}
\label{th-conv-e}
If Assumption A and condition (C) hold, then for any $n\geq 1$, $p\geq 2$ and $T>0$,
\begin{equation}
\label{case:n-arb}
\sup_{(t,x) \in [0,T] \times \bR^d} \bE|H_{n}^{\e}(t,x)-H_{n}(t,x)|^p \to 0 \quad \mbox{as $\e \downarrow 0$,}
\end{equation}
where
\begin{equation}
\label{def-Hn}
H_n(t,x)=I_n^{\circ}\big(f_n(\cdot,x;t) \big)=\sum_{k=0}^{\lfloor n/2 \rfloor} \sum_{\substack{J \subset [n]\\ |J|=n-2k}} \sum_{\substack{\{I_1,\ldots,I_k\} \ {\rm partition} \ {\rm of} \ J^c \\ I_i=\{\ell_i,m_i\} \forall i=1,\ldots,k}} H_{n,k,J,I_1,\ldots,I_k}(t,x)
\end{equation}
and the variable
\begin{align*}
H_{n,k,J,I_1,\ldots,I_k}(t,x)&:=\int_{T_n(t)} \int_{(\bR^d)^{2k}} \left(\int_{(\bR^d)^{n-2k}} \prod_{j=1}^{n}G(t_{j+1}-t_j,x_{j+1}-x_j) \prod_{j\in J}W(dx_j)\right) \\
& \quad \quad \quad \prod_{i=1}^{k}\gamma(x_{\ell_i}-x_{m_i}) d\big((x_j)_{j\in J^c}\big) dt_1 \ldots dt_n
\end{align*}
is well-defined. Here we let $t_{n+1}=t$ and $x_{n+1}=x$.

Consequently, if $v_n(t,x)=1+\sum_{k=1}^{n}H_k(t,x)$, then for any $n \geq 1$, $p\geq 2$ and $T>0$,
\[
\sup_{(t,x) \in [0,T] \times \bR^d} \bE|v_{n}^{\e}(t,x)-v_{n}(t,x)|^p \to 0 \quad \mbox{as $\e \downarrow 0$}.
\]
\end{theorem}

\begin{proof} {\em Step 1.} We consider first the case $n=1$. By the stochastic Fubini's theorem,
\begin{align*}
H_1^{\e}(t,x)&=\int_0^t \int_{\bR^d}G(t-t_1,x-x_1)\dot{W}^{\e}(x_1)dx_1 dt_1\\
&=
\int_{\bR^d}\left(\int_0^t \int_{\bR^d} G(t-t_1,x-x_1) p_{\e}(x_1-y_1)dx_1 dt_1 \right) W(dy_1).
\end{align*}
Letting
\[
H_1(t,x)=\int_0^t \int_{\bR^d}G(t-t_1,x-x_1)W(dx_1)dt_1,
\]
it is not difficult to see that for any $p\geq 2$ and $T>0$,
\begin{equation}
\label{case:n=1}
\sup_{(t,x) \in [0,T] \times \bR^d} \bE|H_{1}^{\e}(t,x)-H_1(t,x)|^p \to 0 \quad \mbox{as $\e \downarrow 0$.}
\end{equation}

{\em Step 2.} We consider now the case $n\geq 2$.
In definition \eqref{def-Hne} of
$H_n^{\e}(t,x)$, we use \eqref{W-product-n} for computing the product of the $n$ Wiener integrals. We obtain the representation:
\begin{equation}
\label{decomp-Hne}
H_n^{\e}(t,x)=\sum_{k=0}^{\lfloor n/2 \rfloor} \sum_{\substack{J \subset [n]\\ |J|=n-2k}} \sum_{\substack{\{I_1,\ldots,I_k\} \ {\rm partition} \ {\rm of} \ J^c \\ I_i=\{\ell_i,m_i\} \forall i=1,\ldots,k}} H_{n,k,J,I_1,\ldots,I_k}^{\e}(t,x),
\end{equation}
where
\begin{align*}
H_{n,k,J,I_1,\ldots,I_k}^{\e}(t,x)&=\int_{T_n(t)} \int_{(\bR^d)^{n}} \prod_{j=1}^{n}G(t_{j+1}-t_j,x_{j+1}-x_j)
I_{n-2k} \big(\bigotimes_{j \in J}p_{\e}(x_j-\cdot) \big)\\
& \quad \quad \quad \prod_{i=1}^{k}(p_{2\e}*\gamma)(x_{\ell_i}-x_{m_i})
d{\bf x} d{\bf t}
\end{align*}
with ${\bf t}=(t_1,\ldots,t_n)$ and ${\bf x}=(x_1,\ldots,x_n)$. Using stochastic Fubini's theorem, we see that
\begin{align*}
H_{n,k,J,I_1,\ldots,I_k}^{\e}(t,x)&=I_{n-2k}\big( A_{n,k,J,I_1,\ldots,I_k}^{\e}(\cdot,t,x)\big) \\ H_{n,k,J,I_1,\ldots,I_k}(t,x)&=I_{n-2k}\big( A_{n,k,J,I_1,\ldots,I_k}(\cdot,t,x)\big),
\end{align*}
where
\begin{align*}
A_{n,k,J,I_1,\ldots,I_k}^{\e}\big((y_j)_{j\in J},t,x\big)&=\int_{T_n(t)} \int_{(\bR^d)^n} \prod_{j=1}^{n}G(t_{j+1}-t_j,x_{j+1}-x_j) \prod_{j\in J}p_{\e}(x_j-y_j)\\
& \quad \quad \quad \prod_{i=1}^{k}
(p_{2\e}*\gamma)(x_{\ell_i}-x_{m_i}) d{\bf x}d{\bf t}\\
A_{n,k,J,I_1,\ldots,I_k}\big((x_j)_{j\in J},t,x\big)&=\int_{T_n(t)} \int_{(\bR^d)^{2k}} \prod_{j=1}^{n}G(t_{j+1}-t_j,x_{j+1}-x_j)  \prod_{i=1}^{k}
\gamma(x_{\ell_i}-x_{m_i}) d\big( (x_j)_{j \in J^c} \big)d{\bf t}.
\end{align*}

We will prove that
\begin{equation}
\label{Ane-An}
\sup_{(t,x) \in [0,T] \times \bR^d}\|A_{n,k,J,I_1,\ldots,I_k}^{\e}(\cdot,t,x)-A_{n,k,J,I_1,\ldots,I_k}
(\cdot,t,x)\|_{\cH^{\otimes (n-2k)}}^2 \to 0 \quad \mbox{as $\e \downarrow 0$}.
\end{equation}
Relation  \eqref{case:n-arb} will follow, since by Minkowksi's inequality and hypercontractivity,
\begin{align*}
& \|H_n^{\e}(t,x)-H_n(t,x)\|_p \leq \sum_{k=0}^{\lfloor n/2 \rfloor} \sum_{\substack{J \subset [n]\\ |J|=n-2k}} \sum_{\substack{\{I_1,\ldots,I_k\} \ {\rm partition} \ {\rm of} \ J^c \\ I_i=\{\ell_i,m_i\} \forall i=1,\ldots,k}} (p-1)^{\frac{n-2k}{2}} \Big((n-2k)!\Big)^{1/2}\\
& \qquad \qquad \qquad \qquad \qquad \qquad
\|A_{n,k,J,I_1,\ldots,I_k}^{\e}(\cdot,t,x)-A_{n,k,J,I_1,\ldots,I_k}
(\cdot,t,x)\|_{\cH^{\otimes (n-2k)}},
\end{align*}
where $\| \cdot\|_p$ is the norm in $L^p(\Omega)$.

{\em Step 3.} We prove that $H_{n,k,J,I_1,\ldots,I_k}(t,x)$
is well-defined, i.e. $A_{n,k,J,I_1,\ldots,I_k}(\cdot,t,x) \in \cH^{\otimes (n-2k)}$.

{\em Step 3.(a)} We first prove that $A_{n,k,J,I_1,\ldots,I_k}(\cdot,t,x) \in L^1((\bR^d)^{n-2k})$, i.e.
\[
{\cal I}:=\int_{(\bR^d)^{n}} \prod_{j=1}^{n}G(t_{j+1}-t_j,x_{j+1}-x_j)
\prod_{i=1}^{k} \gamma(x_{\ell_i}-x_{m_i})d{\bf x}d{\bf t}<\infty.
\]
We denote by $j_1<\ldots<j_{2k}$ the elements of $J^c=\cup_{i=1}^{k}I_i$.
By applying Lemma \ref{semigroup} successively, we can skip the indices  $j\in J$ when estimating ${\cal I}$, so that ${\cal I} \leq (C_{t}^{(1)})^{n-2k}  {\cal I}'$, where
\begin{align*}
 {\cal I}'&= \int_{0<t_{j_1}<\ldots<t_{j_{2k}}<t} \int_{(\bR^d)^{2k}} \prod_{s=1}^{2k} G(t_{j_{s+1}}-t_{j_s}, x_{j_{s+1}}-x_{j_s}) \prod_{i=1}^{k} \gamma(x_{\ell_i}-x_{m_i}) d\big((x_j)_{j\in J^c}\big) d\big((t_j)_{j\in J^c}\big) \\
&=\int_{0<t_{j_1}<\ldots<t_{j_{2k}}<t} \int_{(\bR^d)^{k}} 
\prod_{s=1}^{2k} \cF G(t_{j_{s+1}}-t_{j_s},\cdot)(\xi_{j_1}+\ldots+\xi_{j_s})
\prod_{i=1}^{k} 1_{\{\xi_{\ell_i}=-\xi_{m_i}=\eta_i\}} \prod_{i=1}^{k}\mu(d\eta_i) d\big((t_j)_{j\in J^c}\big),
\end{align*}
where we used Lemma \ref{lemmaF} for the second line.
Note that if $\xi_{\ell_i}=-\xi_{m_i}=\eta_i$ for all $i=1,\ldots,k$, then $\xi_{j_1}+\ldots+\xi_{j_{2k}}=0$.

Suppose that
$\ell_i<m_i$ for all $i=1,\ldots,k$ and $\ell_1<\ldots<\ell_k$.
Say $\ell_i=j_{s_i}$ and $m_i=j_{r_i}$ for all $i=1,\ldots,k$ with $ s_1<\ldots<s_k$. Then $\{s_1,\ldots,s_k\} \cup \{r_1,\ldots,r_k\}=\{1,\ldots,2k\}$. We assume that $\xi_{\ell_i}=-\xi_{m_i}=\eta_i$ for all $i=1,\ldots,k$. Pick $s \in \{1,\ldots,2k\}$. If $s=s_i$ for some $i=1,\ldots,k$ then $\xi_{j_s}=\xi_{\ell_i}=\eta_i$ and we use inequality \eqref{FG-bound}:
\[
|\cF G(t_{j_{s+1}}-t_{j_s},\cdot)(\xi_{j_1}+\ldots+\xi_{j_s}) |\leq C_{t} \left( \frac{1}{1+|\xi_{j_1}+\ldots+\xi_{j_s}|^2}\right)^{1/2}
\]
If $s=r_i$ for some $i=1,\ldots,k$ then $\xi_{j_s}=\xi_{m_i}=-\eta_i$ and we use the inequality
\[
|\cF G(t_{j_{s+1}}-t_{j_s},\cdot)(\xi_{j_1}+\ldots+\xi_{j_s}) |\leq
t_{j_{s+1}}-t_{j_s} \leq t.
\]
Applying Lemma 4.1 of \cite{balan-song}, we obtain:
\begin{align*}
{\cal I}' \leq (tC_t)^k
 \frac{t^{2k}}{(2k)!}\prod_{i=1}^{k} \left( \sup_{z \in \bR^d} \int_{\bR^d} \left( \frac{1}{1+|z+\eta_i|^2}\right)^{1/2} \mu(d\eta_i)\right)=C_t^k
 \frac{t^{3k}}{(2k)!} K_{\mu}^{k}<\infty.
 \end{align*}

{\em Step 3.(b)} We compute the Fourier transform of $A_{n,k,J,I_1,\ldots,I_k}(\cdot,t,x)$.  Recall definition \eqref{def-h} of $h_{{\bf t},(x_j)_{j \in J^c}}$. By Fubini's theorem,
\begin{align*}
& \cF A_{n,k,J,I_1,\ldots,I_k}(\cdot,t,x)\big( (\xi_j)_{j\in J}\big)=\int_{(\bR^d)^{n-2k}} e^{-i \sum_{j\in J} \xi_j \cdot x_j}  A_{n,k,J,I_1,\ldots,I_k}(\cdot,t,x) \big( (x_j)_{j\in J},t,x \big) d\big( (x_j)_{j \in J} \big)\\
&= \quad \int_{(\bR^d)^{n}} e^{-i \sum_{j\in J} \xi_j \cdot x_j}  \int_{T_n(t)} h_{{\bf t},(x_j)_{j \in J^c}}\big((x_j)_{j \in J}\big)
\prod_{i=1}^{k} \gamma(x_{\ell_i}-x_{m_i}) d{\bf t}d{\bf x}\\
&=
\int_{T_n(t)} \int_{(\bR^d)^{2k}}\cF h_{{\bf t},(x_j)_{j \in J^c}}\big( (\xi_j)_{j\in J} \big)
 \prod_{i=1}^{k} \gamma(x_{\ell_i}-x_{m_i}) d\big( (x_j)_{j \in J^c} \big)d{\bf t}.
 \end{align*}
We apply Lemma \ref{lemmaF} to the function $\varphi_{{\bf t},(\xi_j)(j \in J)}((x_j)_{j\in J^c}):=\cF h_{{\bf t},(x_j)_{j \in J^c}}\big( (\xi_j)_{j\in J} \big)$. Then,
\begin{align*}
\cF A_{n,k,J,I_1,\ldots,I_k}(\cdot,t,x)\big( (\xi_j)_{j\in J}\big)&=\int_{T_{n}(t)} \int_{(\bR^d)^{2k}} \varphi_{{\bf t},(\xi_j)(j \in J)}((x_j)_{j\in J^c}) \prod_{i=1}^{k} \gamma(x_{\ell_i}-x_{m_i}) d\big( (x_j)_{j \in J^c} \big)d{\bf t}\\
&=\int_{(\bR^d)^k} \cF \varphi_{{\bf t},(\xi_j)(j \in J)}((\xi_j)_{j\in J^c}) \prod_{i=1}^{k}1_{\{\xi_{\ell_i}=-\xi_{m_i}=\eta_i\}} \prod_{i=1}^{k}\mu(d\eta_i)d{\bf t}.
\end{align*}
Note that
\begin{align*}
\cF \varphi_{{\bf t},(\xi_j)(j \in J)}((\xi_j)_{j\in J^c})&=\int_{(\bR^d)^n} e^{-i \sum_{j=1}^{n} \xi_j \cdot x_j} \prod_{j=1}^{n}G(t_{j+1}-t_j,x_{j+1}-x_j) d{\bf x}\\
&=e^{-i(\sum_{j=1}^n \xi_j) \cdot x} \prod_{j=1}^{n}\cF G(t_{j+1}-t_j,\cdot)(\xi_1+\ldots+\xi_j)
\end{align*}
and
$\sum_{j=1}^n \xi_j=\sum_{j\in J} \xi_j$ if $\xi_{\ell_i}=-\xi_{m_i}=\eta_i$ for all $i=1,\ldots,k$. Hence,
\begin{align}
\nonumber
\cF A_{n,k,J,I_1,\ldots,I_k}(\cdot,t,x)\big( (\xi_j)_{j\in J}\big)&=
e^{-i(\sum_{j\in J} \xi_j) \cdot x}\int_{T_n(t)} \int_{(\bR^d)^k}
\prod_{j=1}^{n}\cF G(t_{j+1}-t_j,\cdot)(\xi_1+\ldots+\xi_j)\\
\label{Fourier-A}
& \quad \quad \quad \prod_{i=1}^{k}1_{\{\xi_{\ell_i}=-\xi_{m_i}=\eta_i\}} \prod_{i=1}^{k}\mu(d\eta_i)d{\bf t}.
\end{align}

{\em Step 3.(c)} We show that the $\cH^{\otimes (n-2k)}$-norm of $A_{n,k,J,I_1,\ldots,I_k}(\cdot,t,x)$ is finite. This argument plays an important role in this paper. By definition,
\[
\|A_{n,k,J,I_1,\ldots,I_k}(\cdot,t,x) \|_{\cH^{\otimes (n-2k)}}^2=\int_{(\bR^d)^{n-2k}}|\cF A_{n,k,J,I_1,\ldots,I_k}(\cdot,t,x)\big( (\xi_j)_{j\in J}\big)|^2 \prod_{j\in J}\mu(d\xi_j).
\]
In relation \eqref{Fourier-A}, we split the product over $j$ into 3 products: for $j \in J$, $j\in \{m_1,\ldots,m_k\}$ and $j \in \{\ell_1,\ldots,\ell_k\}$. For the second product, we use the inequality
\begin{equation}
\label{FG-bound-m}
\prod_{i=1}^{k}|\cF G(t_{m_i+1}-t_{m_i},\cdot)(\xi_1+\ldots+\xi_{m_i})| \leq \prod_{i=1}^{k}(t_{m_i+1}-t_{m_i}).
\end{equation}
Hence,
\begin{align*}
& \|A_{n,k,J,I_1,\ldots,I_k}(\cdot,t,x) \|_{\cH^{\otimes (n-2k)}}^2 \leq \int_{(\bR^d)^{n-2k}} \left( \int_{T_n(t)} \int_{(\bR^d)^k} \prod_{j\in J}|\cF G(t_{j+1}-t_j)(\xi_1+\ldots+\xi_j)| \right. \\
& \left. \prod_{i=1}^{k}(t_{m_i+1}-t_{m_i})  \prod_{i=1}^{k}|\cF G(t_{\ell_i+1}-t_{\ell_i},\cdot)(\xi_1+\ldots+\xi_{\ell_i})| \prod_{i=1}^{k}1_{\{\xi_{\ell_i}=-\xi_{m_i}=\eta_i\}} \prod_{i=1}^{k}\mu(d\eta_i)d{\bf t}\right)^2 \prod_{j\in J}\mu(d\xi_j).
\end{align*}

For each $(\xi_j)_{j \in J}$ fixed, we will apply H\"older's inequality to the measure
\[
\nu_{(\xi_j)_{j \in J}}\big( d{\bf t}, d((\eta_i)_{i=1,\ldots,k})\big) :=
\prod_{i=1}^{k}|\cF G(t_{\ell_i+1}-t_{\ell_i},\cdot)(\xi_1+\ldots+\xi_{\ell_i})| \prod_{i=1}^{k}1_{\{\xi_{\ell_i}=-\xi_{m_i}=\eta_i\}} \prod_{i=1}^{k}\mu(d\eta_i)d{\bf t}.
\]
For this, we need to show that $\nu_{(\xi_j)_{j \in J}}$ is a finite measure on $T_n(t) \times (\bR^d)^k$. By \eqref{FG-bound},
\begin{align*}
& \nu_{(\xi_j)_{j \in J}} \big(T_n(t) \times (\bR^d)^k \big) \leq \int_{T_n(t)} \prod_{i=1}^{k} \left( \sup_{z \in \bR^d} \int_{\bR^d} |\cF G(t_{\ell_i+1}-t_{\ell_i},\cdot)(z+\eta_i)|\mu(d\eta_i)\right) d{\bf t}\\
& \quad \quad \quad  \leq  \frac{t^n}{n!}  \left(C_{t}\sup_{z \in \bR^d} \int_{\bR^d} \left(\frac{1}{1+|z+\eta|^2} \right)^{1/2}\mu(d\eta)\right)^k = \frac{t^n}{n!} C_t^k K_{\mu}^k.
\end{align*}
Hence,
\begin{align*}
& \|A_{n,k,J,I_1,\ldots,I_k}(\cdot,t,x) \|_{\cH^{\otimes (n-2k)}}^2 \leq \frac{t^n}{n!} C_t^k K_{\mu}^k \int_{T_n(t)} \prod_{i=1}^{k}(t_{m_i+1}-t_{m_i})^2  \\
 & \quad \int_{(\bR^d)^{n-2k}}  \int_{(\bR^d)^k} \prod_{j\in J}|\cF G(t_{j+1}-t_j)(\xi_1+\ldots+\xi_j)|^2 \\
 & \quad \prod_{i=1}^{k}|\cF G(t_{\ell_i+1}-t_{\ell_i},\cdot)(\xi_{1}+\ldots+\xi_{\ell_i})|
  \prod_{i=1}^{k}1_{\{\xi_{\ell_i}=-\xi_{m_i}=\eta_i\}} \prod_{i=1}^{k}\mu(d\eta_i) \prod_{j\in J}\mu(d\xi_j)d{\bf t}.
\end{align*}

We have arrived now at the most delicate part of the argument, which requires a careful analysis of the positions of the indices $\ell_i$ and $m_i$. Recall that $\ell_i<m_i$ for all $i=1,\ldots,k$ and $\ell_1<\ldots \ell_k$.
We calculate the sums $\xi_1+\ldots+\xi_j$ successively one after the other and we arrange them in a row-wise manner, one per row. When we encounter an index $j \not \in \{\ell_1,\ldots,\ell_k\} \cup \{m_1,\ldots, m_k\}$, we add $\xi_j$ to the previous sum. When we encounter $j=\ell_i$ for some $i\in \{1,\ldots,k\}$, we add $\eta_i$. When we encounter $j=m_i$ for some $i\in \{1,\ldots,k\}$, we subtract $\eta_i$ and delete this row.
Here is an illustration for $n=11$, $k=4$, $J=\{1,4,6\}$, $I_1=\{2,8\}$, $I_2=\{3,5\}$, $I_3=\{7,10\}$, $I_4=\{9,11\}$:
\begin{center}
\begin{tabular}{l|l|llllllllllll}
          & & $\xi_1$ & & & & & & & & & & &\\
& $\ell_1=2$ & $\xi_1$ & $+\eta_1$ & & & & & & & & & &\\
& $\ell_2=3$ & $\xi_1$ & $+\eta_1$ & $+\eta_2$  & & & & & & & & &\\
& & $\xi_1$ & $+\eta_1$ & $+\eta_2$ & $+\xi_4$ & & & & & & &\\
delete  & $m_2=5$ & $\xi_1$ & $+\eta_1$ & $+\eta_2$ & $+\xi_4$ & $-\eta_2$ & & & & & &\\
& & $\xi_1$ & $+\eta_1$ & & $+\xi_4$ & & $+\xi_6$ & & & & & &\\
& $\ell_3=7$  & $\xi_1$ & $+\eta_1$ & & $+\xi_4$ & & $+\xi_6$ & $+\eta_3$ & & & & &\\
delete  & $m_1=8$  & $\xi_1$ & $+\eta_1$ & & $+\xi_4$ & & $+\xi_6$ & $+\eta_3$ & & $-\eta_1$ & & \\
& $\ell_4=9$  & $\xi_1$ &  & & $+\xi_4$ & & $+\xi_6$ & $+\eta_3$ & & & $+\eta_4$ & & \\
delete  & $m_3=10$  & $\xi_1$ &  & & $+\xi_4$ & & $+\xi_6$ & $+\eta_3$ & & & $+\eta_4$ & $-\eta_3$ & \\
delete  & $m_4=11$  & $\xi_1$ &  & & $+\xi_4$ & & $+\xi_6$ &  & & & $+\eta_4$ &  & $-\eta_4$
\end{tabular}
\end{center}
When we move from one row to the next one, we will always have {\em just one new term} added to the sum, and this term can be either some $\eta_i$ with $i\in \{1,\ldots,k\}$ which did not appear in the rows above, or some $\xi_j$ with $j \in J$ which did not appear in the rows above. The problematic rows corresponding to $j \in \{m_1,\ldots,m_k\}$ have been deleted, due to our use of inequality \eqref{FG-bound-m}.

We decompose $\{1,\ldots,\ell_i\}$ into the blocks
$B_s=\{j;\ell_{s-1}+1\leq j \leq \ell_s\}$ for $s=1,\ldots,k$, where
$\ell_0=0$. Then we write
\[
\sum_{j=1}^{\ell_i}\xi_j=\sum_{s=1}^{i} (\xi_{\ell_{s-1}+1} \ldots+\xi_{\ell_s-1}+\eta_s).
\]
The sum $\sum_{j=1}^{\ell_i}\xi_j$ will always contain the term $\eta_i$. (This sum may or may not contain $\eta_s$ for $s< i$,
depending on whether $m_s>\ell_i$ or $m_s<\ell_i$. If $m_s<\ell_i$, then  one of the blocks $B_{s+1},\ldots,B_i$ will contain $m_s$, and the term $-\eta_s$ added at the position corresponding to $m_s$ will cancel out with the term $\eta_s$ at the end of block $B_s$.)

Based on this procedure and using inequality \eqref{FG-bound} and Lemma 4.1 of \cite{balan-song}, we obtain: 
\begin{align}
\nonumber
& \|A_{n,k,J,I_1,\ldots,I_k}(\cdot,t,x) \|_{\cH^{\otimes (n-2k)}}^2 \leq
 \frac{t^n}{n!} K_{\mu}^{k} C_t^k  \int_{T_n(t)} d{\bf t} \prod_{i=1}^{k}(t_{m_i+1}-t_{m_i})^2 \\
\nonumber
& \qquad \qquad \qquad \qquad \qquad
 \prod_{j \in J} \left(\sup_{z \in \bR^d}\int_{\bR^d} |\cF G (t_{j+1}-t_j,\cdot)(z+\xi_j)|^2 \mu(d\xi_j) \right) \\
 \nonumber
 & \qquad \qquad \qquad \qquad \qquad 
 \prod_{i=1}^{k} \left( \sup_{z \in \bR^d} \int_{\bR^d} |\cF G(t_{\ell_i+1}-t_{\ell_i},\cdot)(z+\eta_i)|\mu(d\eta_i)\right) \\
\nonumber
 &\leq \frac{t^n}{n!} K_{\mu}^{k}C_t^k   \left(\int_{T_n(t)} \prod_{i=1}^{k}(t_{m_i+1}-t_{m_i})^2 d{\bf t}  \right) \left(C_{t}^2 \sup_{z \in \bR^d}\int_{\bR^d} \frac{1}{1+|z+\xi|^2}\mu(d\xi) \right)^{n-2k} \\
\nonumber
& \qquad \qquad \qquad \qquad \qquad 
 \left(C_{t} \sup_{z \in \bR^d} \int_{\bR^d} \left(\frac{1}{1+|z+\eta|^2}\right)^{1/2}\mu(d\eta)\right)^k \\
 \label{DCT-justif}
 & = \frac{t^n}{n!}K_{\mu}^n C_{t}^{2(n-k)}\frac{2^k t^{n+2k}}{(n+2k)!}.
\end{align}

\medskip

{\em Step 4.} We prove \eqref{Ane-An}. For this, we will use the decomposition:
\begin{equation}
\label{An-Ane-decomp}
A_{n,k,J,I_1,\ldots,I_k}(\cdot,t,x)-A_{n,k,J,I_1,\ldots,I_k}^{\e}(\cdot,t,x)=
T_{n,k,J,I_1,\ldots,I_k}^{\e}(\cdot,t,x)+T_{n,k,J,I_1,\ldots,I_k}'^{\e}
(\cdot,t,x),
\end{equation}
where
\begin{align}
\label{def-Te}
T_{n,k,J,I_1,\ldots,I_k}^{\e}\big((y_j)_{j\in J},t,x \big)&=\int_{T_n(t)} \int_{(\bR^d)^{2k}} \big(h_{{\bf t},(x_j)_{j \in J^c}}^{(t,x)}-h_{{\bf t},(x_j)_{j \in J^c}}^{(t,x)}* p_{\e}^{\otimes (n-2k)} \big)\big( (y_j)_{j\in J} \big)\\
\nonumber
& \quad \quad 
\prod_{i=1}^{k} (p_{2\e}* \gamma)(x_{\ell_i}-x_{m_i}) d\big( (x_j)_{j \in J^c} \big) d{\bf t} \\
\label{def-Te'}
T_{n,k,J,I_1,\ldots,I_k}'^{\e}\big((y_j)_{j\in J},t,x\big)&=\int_{T_n(t)} \int_{(\bR^d)^{2k}} h_{{\bf t},(x_j)_{j \in J^c}}^{(t,x)} \big( (y_j)_{j\in J} \big)\\
& \quad \quad 
\nonumber
\Big(\prod_{i=1}^{k} \gamma(x_{\ell_i}-x_{m_i})-
 \prod_{i=1}^{k} (p_{2\e}* \gamma)(x_{\ell_i}-x_{m_i}) \Big) d\big( (x_j)_{j \in J^c} \big) d{\bf t}.
\end{align}

We treat separately the two terms. Recall that $p_{2\e}* \gamma=\cF \mu_{\e}$ where $\mu_{\e}$ is given by \eqref{def-mu-e}.
Similarly to Step 3, it can be proved that the functions $T_{n,k,J,I_1,\ldots,I_k}^{\e}(\cdot,t,x)$ and $T_{n,k,J,I_1,\ldots,I_k}'^{\e}(\cdot,t,x)$ belong to $L^1((\bR^d)^{n-2k})$ and their Fourier transforms are given by:
\begin{align}
\nonumber
& \cF T_{n,k,J,I_1,\ldots,I_k}^{\e}(\cdot,t,x) \big( (\xi_j)_{j \in J}\big) = e^{-i (\sum_{j \in J} \xi_j) \cdot x} \big(1-e^{-\frac{\e}{2} \sum_{j\in J}|\xi_j|^2}\big) \\
\label{def-FTe}
& \quad \quad \quad \int_{T_n(t)} \int_{(\bR^d)^k} \prod_{j=1}^{n} \cF G(t_{j+1}-t_j,\cdot)\big(\sum_{s=1}^{j}\xi_s \big) \prod_{i=1}^{k} 1_{\{\xi_{\ell_i}=-\xi_{m_i}=\eta_i\}}  \mu_{\e}(d\eta_i) d{\bf t} \\
\nonumber
& \cF T_{n,k,J,I_1,\ldots,I_k}'^{\e}(\cdot,t,x) \big( (\xi_j)_{j \in J}\big) = e^{-i (\sum_{j \in J} \xi_j) \cdot x} \int_{T_n(t)} \int_{(\bR^d)^k} \prod_{j=1}^{n} \cF G(t_{j+1}-t_j,\cdot)\big(\sum_{s=1}^{j}\xi_s \big) \\
\label{def-F'Te}
& \quad \quad \quad  \prod_{i=1}^{k} 1_{\{\xi_{\ell_i}=-\xi_{m_i}=\eta_i\}} \Big(1-e^{-\e \sum_{i=1}^{k}|\eta_i|^2}\Big) \prod_{i=1}^{k} \mu(d\eta_i) d{\bf t}.
\end{align}
(In these calculations, we have applied Lemma \ref{lemmaF} to the function $p_{2\e}*\gamma=\cF \mu_{\e}$.)

By applying H\"older's inequality and \eqref{FG-bound} (as in Step 3), we obtain:
\begin{align}
\nonumber
& \|T_{n,k,J,I_1,\ldots,I_k}^{\e}(\cdot,t,x) \|_{\cH^{\otimes (n-2k)}}^2 \leq \frac{t^n}{n!} C_t^k K_{\mu}^k
\int_{T_n(t)} \prod_{i=1}^{k}(t_{m_i+1}-t_{m_i})^2  \\
\nonumber
& \quad \quad \int_{(\bR^d)^{n-2k}}  \big(1-e^{-\frac{\e}{2} \sum_{j\in J}|\xi_j|^2}\big)^2 \int_{(\bR^d)^k} \prod_{j\in J}|\cF G(t_{j+1}-t_j)(\sum_{s=1}^j \xi_s)|^2 \\
\nonumber
&   \quad \quad  \prod_{i=1}^{k}|\cF G(t_{\ell_i+1}-t_{\ell_i},\cdot)(\sum_{s=1}^{\ell_i}\xi_{s})| \prod_{i=1}^{k}1_{\{\xi_{\ell_i}=-\xi_{m_i}=\eta_i\}} \prod_{i=1}^{k}\mu(d\eta_i) \prod_{j\in J}\mu(d\xi_j)d{\bf t} \\
\nonumber
& \leq \frac{t^n}{n!} C_t^{2(n-k)} K_{\mu}^k   \int_{T_n(t)} \prod_{i=1}^{k}(t_{m_i+1}-t_{m_i})^2 \\
\nonumber
& \quad \quad  \int_{(\bR^d)^{n-2k}}  \big(1-e^{-\frac{\e}{2} \sum_{j\in J}|\xi_j|^2}\big)^2 \int_{(\bR^d)^k} \prod_{j\in J}
\frac{1}{1+|\sum_{s=1}^{j}\xi_s|^2}  \\
\label{norm-Te}
&  \quad \quad  \prod_{i=1}^{k} \left(\frac{1}{1+|\sum_{s=1}^{\ell_i}\xi_{s}|^2}\right)^{1/2} \prod_{i=1}^{k}1_{\{\xi_{\ell_i}=-\xi_{m_i}=\eta_i\}} \prod_{i=1}^{k}\mu(d\eta_i) \prod_{j\in J}\mu(d\xi_j)d{\bf t}
\end{align}
and
\begin{align*}
& \|T_{n,k,J,I_1,\ldots,I_k}'^{\e}(\cdot,t,x) \|_{\cH^{\otimes (n-2k)}}^2 \leq \frac{t^n}{n!} C_t^k K_{\mu}^k
\int_{T_n(t)} \prod_{i=1}^{k}(t_{m_i+1}-t_{m_i})^2  \\
& \quad \quad \int_{(\bR^d)^{n-2k}}  \int_{(\bR^d)^k} \prod_{j\in J}|\cF G(t_{j+1}-t_j)(\sum_{s=1}^j \xi_s)|^2 \prod_{i=1}^{k}|\cF G(t_{\ell_i+1}-t_{\ell_i},\cdot)(\sum_{s=1}^{\ell_i}\xi_{s})| \\
&   \quad \quad
\Big(1-e^{-\e \sum_{i=1}^{k}|\eta_i|^2}\Big)^2
\prod_{i=1}^{k}1_{\{\xi_{\ell_i}=-\xi_{m_i}=\eta_i\}} \prod_{i=1}^{k}\mu(d\eta_i) \prod_{j\in J}\mu(d\xi_j)d{\bf t} \\
& \leq \frac{t^n}{n!} C_t^{2(n-k)} K_{\mu}^k   \int_{T_n(t)} \prod_{i=1}^{k}(t_{m_i+1}-t_{m_i})^2 \\
& \quad \quad \int_{(\bR^d)^{n-2k}}   \int_{(\bR^d)^k} \prod_{j\in J}
\frac{1}{1+|\sum_{s=1}^{j}\xi_s|^2}  \prod_{i=1}^{k} \left(\frac{1}{1+|\sum_{s=1}^{\ell_i}\xi_{s}|^2}\right)^{1/2}  \\
&  \quad \quad
\Big(1-e^{-\e \sum_{i=1}^{k}|\eta_i|^2}\Big)^2\prod_{i=1}^{k}1_{\{\xi_{\ell_i}=-\xi_{m_i}=\eta_i\}} \prod_{i=1}^{k}\mu(d\eta_i) \prod_{j\in J}\mu(d\xi_j)d{\bf t}.
\end{align*}

By the dominated convergence theorem, both
$\|T_{n,k,J,I_1,\ldots,I_k}^{\e}(\cdot,t,x)\|_{\cH^{\otimes (n-2k)}}$ and \linebreak $\|T_{n,k,J,I_1,\ldots,I_k}'^{\e}(\cdot,t,x)\|_{\cH^{\otimes (n-2k)}}$ converge to 0 as $\e \downarrow 0$, uniformly in $(t,x) \in [0,T] \times \bR^d$. The application of this theorem is justified by \eqref{DCT-justif}. 

This proves \eqref{Ane-An} and concludes the proof of the theorem.
\end{proof}

\subsection{Convergence as $n \to \infty$}
\label{section-n-conv}

In this section, we examine the two vertical sides of the diagram above, involving the convergence in $n$.

\begin{theorem}
\label{th-n-conv}
If Assumption A and condition (C) hold, then for any $p\geq 2$ and $T>0$, the following limit
\[
v(t,x):=\lim_{n \to \infty} v_n(t,x) \quad \mbox{exists in $L^p(\Omega)$,}
\]
uniformly in $[0,T] \times \bR^d$, where $v_n(t,x)$ is defined in Theorem \ref{th-conv-e}.
\end{theorem}

\begin{proof} We will prove that $\{v_n(t,x)\}_{n\geq 0}$ is a Cauchy sequence in $L^p(\Omega)$, uniformly in $(t,x) \in [0,T] \times \bR^d$. For this, it is enough to prove that
\begin{equation}
\label{sum-Hn}
\sum_{n\geq 0} \sup_{(t,x) \in [0,T] \times \bR^d} \|H_n(t,x)\|_p < \infty.
\end{equation}
Recall definition \eqref{def-Hn} of $H_n(t,x)$. By Minkowski's inequality and hypercontractivity,
\begin{align*}
& \|H_n(t,x)\|_p \leq \\
& \quad \sum_{k=0}^{\lfloor n/2 \rfloor} \sum_{\substack{J \subset [n]\\ |J|=n-2k}} \sum_{\substack{\{I_1,\ldots,I_k\} \ {\rm partition} \ {\rm of} \ J^c \\ I_i=\{\ell_i,m_i\} \forall i=1,\ldots,k}} (p-1)^{\frac{n-2k}{2}} \big((n-2k)!\big)^{1/2}
\|A_{n,k,J,I_1,\ldots,I_k}
(\cdot,t,x)\|_{\cH^{\otimes (n-2k)}}.
\end{align*}

We now use \eqref{DCT-justif}. Since there are $\binom{n}{2k}$ sets $J \subset [n]$ with $|J|=n-2k$, and for each set $J$ there are
$\frac{(2k)!}{2^k k!}$ (unordered)
partitions $\{I_1,\ldots,I_k\}$ of $J^c$ with $|I_i|=2$ for all $i=1,\ldots,k$,
\begin{align*}
\|H_n(t,x)\|_p & \leq \sum_{k=0}^{\lfloor n/2 \rfloor} \frac{n!}{(n-2k)! \, 2^k k!} (p-1)^{\frac{n-2k}{2}} \big((n-2k)!\big)^{1/2} \left[\frac{t^n}{n!} K_{\mu}^n C_t^{2(n-k)}
\frac{2^k t^{n+2k}}{(n+2k)!}\right]^{1/2} \\
& \leq C_{p,\mu,t}^{n/2} \sum_{k=0}^{\lfloor n/2 \rfloor} \frac{(2n)! }{(n-2k)! \, k! \, (n+k)! } \cdot \frac{n!}{(2n)!}\cdot \left(\frac{(n-2k)!}{n!}\right)^{1/2} \big((n+k)!\big)^{1/2} ,
\end{align*}
where $C_{p,\mu,t}=(p-1) K_{\mu}C_t^2 t (t \vee 1)^{3/2}$ and we used the fact that $(n+2k)! \geq (n+k)!$. By Stirling's formula, $(2n)! \geq c_1^{-n} (n!)^2$ and $(n+k)!=\Gamma(n+k+1)\leq \Gamma(\frac{3}{2}n+1) \leq c_2^n (n!)^{3/2}$, where $c_1>0$ and $c_2>0$ are positive constants. Since $(n-2k)! \leq n!$, we obtain:
\begin{align*}
\sup_{(t,x) \in [0,T] \times \bR^d}\|H_n(t,x)\|_p &
\leq C_{p,\mu,T}^{n/2} \sum_{k=0}^{\lfloor n/2 \rfloor} \binom{2n}{n-2k,k,n+k}\cdot \frac{c_1^{n}}{n!} \cdot c_2^{n/2}(n!)^{3/4}  \\
&\leq  C_{p,\mu,T}^{n/2}\, c_1^n c_2^{n/2} \frac{1}{(n!)^{1/4}} 3^{2n}.
\end{align*}
This proves \eqref{sum-Hn}.
\end{proof}

We now prove the uniform convergence in $\e$ which was mentioned in Section \ref{section-ve}.

\begin{theorem}
If Assumption A and condition (C) holds, then for any $p\geq 2$ and $T>0$,
\[
\lim_{n \to \infty} v_n^{\e}(t,x)=v^{\e}(t,x) \ \mbox{in $L^p(\Omega)$},
\]
uniformly in $[0,T] \times \bR^d$ and $\e>0$, where $v^{\e}$ is the solution of equation \eqref{wave-m}.
\end{theorem}

\begin{proof} As in Theorem \ref{th-n-conv}, it is enough to prove that
\[
\sum_{n\geq 0} \sup_{(t,x) \in [0,T] \times \bR^d} \sup_{\e>0}\|H_n^{\e}(t,x)\|_p < \infty.
\]
Recall decomposition \eqref{decomp-Hne} of $H_n^{\e}(t,x)$ and the fact that \[
H_{n,k,J,I_1,\ldots,I_k}^{\e}(t,x)=I_{n-2k}
\big(A_{n,k,J,I_1,\ldots,I_k}^{\e}(\cdot,t,x)\big).
\]
Note that $A_{n,k,J,I_1,\ldots,I_k}^{\e}(\cdot,t,x)$ has the same expression as
$T_{n,k,J,I_1,\ldots,I_k}^{\e}(\cdot,t,x)$ (defined by \eqref{def-Te}), but with $h_{{\bf t},(x_j)_{j\in J^c}}^{(t,x)}-h_{{\bf t},(x_j)_{j\in J^c}}^{(t,x)} * p_{\e}^{\otimes (n-2k)}$ replaced by $h_{{\bf t},(x_j)_{j\in J^c}}^{(t,x)}* p_{\e}^{\otimes (n-2k)}$. Similarly to
\eqref{norm-Te}, we have:
\begin{align*}
& \|A_{n,k,J,I_1,\ldots,I_k}^{\e}(\cdot,t,x) \|_{\cH^{\otimes (n-2k)}}^2
\leq \frac{t^n}{n!} C_t^{2(n-k)} K_{\mu}^k   \int_{T_n(t)} \prod_{i=1}^{k}(t_{m_i+1}-t_{m_i})^2  \int_{(\bR^d)^{n-2k}}  e^{-\e \sum_{j\in J}|\xi_j|^2}\\
&  \quad \int_{(\bR^d)^k} \prod_{j\in J}
\frac{1}{1+|\sum_{s=1}^j \xi_s|^2}   \prod_{i=1}^{k} \left(\frac{1}{1+|\sum_{s=1}^{\ell_i}\xi_s|^2}\right)^{1/2} \prod_{i=1}^{k}1_{\{\xi_{\ell_i}=-\xi_{m_i}=\eta_i\}} \prod_{i=1}^{k}\mu_{\e}(d\eta_i) \prod_{j\in J}\mu(d\xi_j)d{\bf t}.
\end{align*}
Using the bound $e^{-x} \leq 1$ for $x\geq 0$, and proceeding as for \eqref{DCT-justif}, we obtain:
\[
\|A_{n,k,J,I_1,\ldots,I_k}^{\e}(\cdot,t,x) \|_{\cH^{\otimes (n-2k)}}^2 \leq \frac{t^n}{n!} K_{\mu}^{n} C_{t}^{2(n-k)}\frac{2^k t^{n+2k}}{(n+2k)!}.
\]
The rest of the proof is the same as for Theorem \ref{th-n-conv}.
\end{proof}

\subsection{Final step: $v$ is Stratonovich solution}
\label{section-Stratonovich}

In this section, we give the proof of the last statement in Theorem \ref{main}.

\begin{theorem}
If Assumption A and condition (C) hold, then $v$ is a Stratonovich solution to equation \eqref{wave}, i.e. for any $t>0$ and $x \in \bR^d$,
\[
v(t,x)\stackrel{P}{=}\lim_{\e \downarrow 0}\left(1+\int_0^t \int_{\bR^d}G(t-s,x-y)v(s,y)\dot{W}^{\e}(y)dyds\right).
\]
\end{theorem}

\begin{proof}
{\em Step 1.} Since $v^{\e}$ satisfies \eqref{wave-m-1}, we have the decomposition
\begin{equation}
\label{T-decomp}
T^{\e}(t,x)=1+\int_0^t \int_{\bR^d}G(t-s,x-y)v(s,y)\dot{W}^{\e}(y)dyds=v^{\e}(t,x)+L^{\e}(t,x),
\end{equation}
where
\[
L^{\e}(t,x)=\int_0^t \int_{\bR^d}G(t-s,x-y)\big(v(s,y)-v^{\e}(s,y) \big) \dot{W}^{\e}(y)dyds.
\]
Since $v^{\e}(t,x) \to v(t,x)$ in $L^p(\Omega)$ as $\e \downarrow 0$, it suffices to prove that $L^{\e}(t,x) \stackrel{P}{\to}0$ as $\e \downarrow 0$. In fact, we will show that:
\begin{equation}
\label{Le-conv}
\bE|L^{\e}(t,x)| \to 0 \quad {\mbox as} \ \e \downarrow 0.
\end{equation}

{\em Step 2.} We calculate $L^{\e}(t,x)$. Note that
\[
v(s,y)-v^{\e}(s,y)=\sum_{n\geq 1}\big(H_n(s,y)-H_{n}^{\e}(s,y)\big).
\]
We use relations \eqref{def-Hn} and \eqref{decomp-Hne} for expressing $H_n(s,y)$ and $H_n^{\e}(s,y)$. We obtain:
\begin{align*}
v(s,y)-v^{\e}(s,y)=
 \sum_{n\geq 1}\sum_{k=0}^{\lfloor n/2 \rfloor} \sum_{\substack{J \subset [n]\\ |J|=n-2k}} \sum_{\substack{\{I_1,\ldots,I_k\} \ {\rm partition} \ {\rm of} \ J^c \\ I_i=\{\ell_i,m_i\} \forall i=1,\ldots,k}}
I_{n-2k}\big(B_{n,k,J,I_1,\ldots,I_k}^{\e}(\cdot,s,y)\big),
\end{align*}
where $B_{n,k,J,I_1,\ldots,I_k}^{\e}(\cdot,s,y)=A_{n,k,J,I_1,\ldots,I_k}(\cdot,s,y)-
A_{n,k,J,I_1,\ldots,I_k}^{\e}(\cdot,s,y)$.

We multiply the last equation on display by $\dot{W}^{\e}(y)=I_1\big(p_{\e}(y-\cdot)\big)$. We obtain:
\begin{align}
\nonumber
& \big(v(s,y)-v^{\e}(s,y)\big) \dot{W}^{\e}(y)=\\
\label{ve-v-W}
& \quad
 \sum_{n\geq 1}\sum_{k=0}^{\lfloor n/2 \rfloor} \sum_{\substack{J \subset [n]\\ |J|=n-2k}} \sum_{\substack{\{I_1,\ldots,I_k\} \ {\rm partition} \ {\rm of} \ J^c \\ I_i=\{\ell_i,m_i\} \forall i=1,\ldots,k}}
I_{n-2k}\big(B_{n,k,J,I_1,\ldots,I_k}^{\e}(\cdot,s,y)\big)
I_1\big(p_{\e}(y-\cdot)\big).
\end{align}
We apply the product formula \eqref{product} for computing the product of the two Wiener integrals:
\begin{align*}
I_{n-2k}\big(B_{n,k,J,I_1,\ldots,I_k}^{\e}(\cdot,s,y)\big) I_1\big(p_{\e}(y-\cdot)\big)&=I_{n-2k+1}\big( \widetilde{B}_{n,k,J,I_1,\ldots,I_k}^{\e}(\cdot,s,y) \otimes p_{\e}(y-\cdot) \big)\\
& + (n-2k)I_{n-2k-1}\big( \widetilde{B}_{n,k,J,I_1,\ldots,I_k}^{\e}(\cdot,s,y) \otimes_1 p_{\e}(y-\cdot) \big),
\end{align*}
where $\widetilde{B}_{n,k,J,I_1,\ldots,I_k}^{\e}(\cdot,s,y)$ is the symmetrization of $B_{n,k,J,I_1,\ldots,I_k}^{\e}(\cdot,s,y)$.

Note that for any functions $f \in \cH^{\otimes m}$ and $g \in \cH$,
\[
I_{m}(\widetilde{f} \otimes g)=I_{m}(f \otimes g),
\]
since
the symmetrization of $\widetilde{f} \otimes g$ coincides with the symmetrization of $f \otimes g$, and
\[
mI_{m-1}(\widetilde{f} \otimes_1 g)=\sum_{\ell=1}^{m}\int_{(\bR^d)^{m-1}} \langle f(x_1,\ldots,x_{\ell},\cdot,x_{\ell+1},\ldots,x_m),g \rangle_{\cH} \prod_{\substack{j=1 \\ j \not = \ell}}^{m}W(dx_j).
\]
Hence,
\begin{align}
\nonumber
& I_{n-2k}\big(B_{n,k,J,I_1,\ldots,I_k}^{\e}(\cdot,s,y)\big) I_1\big(p_{\e}(y-\cdot)\big)=I_{n-2k+1}\big( B_{n,k,J,I_1,\ldots,I_k}^{\e}(\cdot,s,y) \otimes p_{\e}(y-\cdot) \big)+ \\
\label{prod-2-int}
& \quad \quad \quad  \sum_{\ell \in J} \int_{(\bR^d)^{n-2k-1}}\langle B_{n,k,J,I_1,\ldots,I_k}^{\e}((x_j)_{\substack{j \in J \\ j \not=\ell}},\cdot,s,y),p_{\e}(y-\cdot) \rangle_{\cH} \prod_{\substack{j\in J \\ j \not = \ell}}^{m}W(dx_j),
\end{align}
where $B_{n,k,J,I_1,\ldots,I_k}^{\e}((x_j)_{\substack{j \in J \\ j \not=\ell}},\cdot,s,y)$ denotes the function $x_{\ell} \mapsto B_{n,k,J,I_1,\ldots,I_k}^{\e}((x_j)_{j \in J},s,y)$.

We introduce expression \eqref{prod-2-int} into \eqref{ve-v-W}. Then we multiply by $G(t-s,x-y)$ and we integrate $ds dy$ on $[0,t] \times \bR^d$. Using stochastic Fubini's theorem for interchanging the multiple Wiener integral with the $ds dy$ integral, we obtain that:
\begin{align*}
L^{\e}(t,x)&=\sum_{n\geq 1}\sum_{k=0}^{\lfloor n/2 \rfloor} \sum_{\substack{J \subset [n]\\ |J|=n-2k}} \sum_{\substack{\{I_1,\ldots,I_k\} \ {\rm partition} \ {\rm of} \ J^c \\ I_i=\{\ell_i,m_i\} \forall i=1,\ldots,k}} \Big(
\cI_{n,k,J,I_1,\ldots,I_k}^{\e}(t,x)+ \cJ_{n,k,J,I_1,\ldots,I_k}^{\e}(t,x) \Big),
\end{align*}
where
\begin{align*}
\cI_{n,k,J,I_1,\ldots,I_k}^{\e}(t,x)&=\int_{(\bR^d)^{n-2k+1}}F_{n,k,J,I_1,\ldots,I_k}^{\e}
\big((x_j)_{j \in J},z,t,x \big)  \prod_{j\in J}W(dx_j)W(dz)\\
F_{n,k,J,I_1,\ldots,I_k}^{\e}\big((x_j)_{j \in J},z,t,x \big)&=\int_0^t \int_{\bR^d}G(t-s,x-y) B_{n,k,J,I_1,\ldots,I_k}^{\e}((x_j)_{j \in J},s,y) p_{\e}(y-z)dyds
\end{align*}
and
\begin{align*}
\cJ_{n,k,J,I_1,\ldots,I_k}^{\e}(t,x)&=\sum_{\ell \in J} \cJ_{n,k,J,I_1,\ldots,I_k}^{\e,\ell}(t,x)\\
\cJ_{n,k,J,I_1,\ldots,I_k}^{\e,\ell}(t,x)&=\int_{(\bR^d)^{n-2k-1}} Q_{n,k,J,I_1,\ldots,I_k}^{\e,\ell} \big((x_j)_{\substack{j \in J \\ j \not=\ell}},t,x \big)
\prod_{\substack{j \in J \\ j \not= \ell}}W(dx_j)\\
Q_{n,k,J,I_1,\ldots,I_k}^{\e,\ell} \big((x_j)_{\substack{j \in J \\ j \not=\ell}},t,x \big)&=\int_0^t \int_{\bR^d}G(t-s,x-y) \langle B_{n,k,J,I_1,\ldots,I_k}^{\e}\big( (x_j)_{\substack{j \in J \\ j \not=\ell}},\cdot,s,y \big),p_{\e}(y-\cdot)\rangle_{\cH}dyds.
\end{align*}

Since $B^{\e}=T^{\e}+T'^{\e}$ (see \eqref{An-Ane-decomp}), we write $F^{\e}=G^{\e}+G'^{\e}$ and $Q^{\e,\ell}=R^{\e,\ell}+R'^{\e,\ell}$, where
\begin{align*}
G_{n,k,J,I_1,\ldots,I_k}^{\e} \big((x_j)_{j \in J},z,t,x \big)&=\int_0^t \int_{\bR^d} G(t-s,x-y) T_{n,k,J,I_1,\ldots,I_k}^{\e} \big((x_j)_{j \in J},s,y \big)p_{\e}(y-z) dyds \\
R_{n,k,J,I_1,\ldots,I_k}^{\e,\ell} \big((x_j)_{\substack{j \in J \\ j \not=\ell}},t,x \big)&=\int_0^t \int_{\bR^d} G(t-s,x-y) \langle T_{n,k,J,I_1,\ldots,I_k}^{\e} \big((x_j)_{\substack{j \in J \\ j\not=\ell}},\cdot,s,y \big),p_{\e}(y-\cdot) \rangle_{\cH} dyds,
\end{align*}
and $G'^{\e},R'^{\e,\ell}$ have the same form as $G^{\e}$, respectively $R^{\e}$, but with $T^{\e}$ replaced by $T'^{\e}$.

We obtain the decomposition:
\begin{align*}
L^{\e}(t,x)&=\sum_{n\geq 1}\sum_{k=0}^{\lfloor n/2 \rfloor} \sum_{\substack{J \subset [n]\\ |J|=n-2k}} \sum_{\substack{\{I_1,\ldots,I_k\} \ {\rm partition} \ {\rm of} \ J^c \\ I_i=\{\ell_i,m_i\} \forall i=1,\ldots,k}} \Big(
\cG_{n,k,J,I_1,\ldots,I_k}^{\e}(t,x)+ \cG_{n,k,J,I_1,\ldots,I_k}'^{\e}(t,x) +\\
& \qquad \qquad \qquad \qquad \qquad \qquad \qquad \qquad \quad
\cR_{n,k,J,I_1,\ldots,I_k}^{\e}(t,x)+ \cR_{n,k,J,I_1,\ldots,I_k}'^{\e}(t,x) \Big)\\
&=:L_{1}^{\e}(t,x)+L_{2}^{\e}(t,x)+L_{3}^{\e}(t,x)+L_{4}^{\e}(t,x),
\end{align*}
where
\begin{align*}
\cG_{n,k,J,I_1,\ldots,I_k}^{\e}(t,x)&=\int_{(\bR^d)^{n-2k+1}}G_{n,k,J,I_1,\ldots,I_k}^{\e}
\big((x_j)_{j \in J},z,t,x \big)  \prod_{j\in J}W(dx_j)W(dz) \\
\cR_{n,k,J,I_1,\ldots,I_k}^{\e}(t,x)&=\sum_{\ell \in J}\cR_{n,k,J,I_1,\ldots,I_k}^{\e,\ell}(t,x) \\
\cR_{n,k,J,I_1,\ldots,I_k}^{\e,\ell}(t,x)&=\int_{(\bR^d)^{n-2k-1}} R_{n,k,J,I_1,\ldots,I_k}^{\e,\ell} \big((x_j)_{\substack{j \in J \\ j \not=\ell}},t,x \big)
\prod_{\substack{j \in J \\ j \not= \ell}}W(dx_j),
\end{align*}
$\cG'^{\e}$ has the same form as $\cG^{\e}$ but with $G^{\e}$ replaced by $G'^{\e}$, $\cR'^{\e}=\sum_{\ell \in J} \cR'^{\e,\ell}$ and $\cR'^{\e,\ell}$ has the same form as $\cR^{\e,\ell}$ but with $R^{\e,\ell}$ replaced by $R'^{\e,\ell}$.

We will prove that
\begin{equation}
\label{Li-e-zero}
\bE|L_{i}^{\e}(t,x)| \to 0 \quad \mbox{as} \ \e \downarrow 0, \quad \mbox{for} \ i=1,2,3,4.
\end{equation}

{\em Step 3.} We treat $L_1^{\e}(t,x)$.
We proceed similarly to Step 3 of the proof of Theorem \ref{th-conv-e}. Note that the function $G_{n,k,J,I_1,\ldots,I_k}^{\e}(\cdot,t,x)\in L^1((\bR^d)^{n-2k+1})$. Letting $t_{n+1}=s$ and ${\bf t}=(t_1,\ldots,t_{n+1})$, it can be proved that
\begin{align*}
& \cF G_{n,k,J,I_1,\ldots,I_k}^{\e}(\cdot,t,x)\big((\xi_j)_{j \in J},\eta \big)=e^{-i (\sum_{j\in J} \xi_j) \cdot x}e^{-i \eta \cdot x} e^{-\frac{\e}{2}|\eta|^2}
\Big(1-e^{-\frac{\e}{2}\sum_{j\in J}|\xi_j|^2} \Big) \\
& \quad \int_{T_{n+1}(t)} \int_{(\bR^d)^k}\cF G(t-t_{n+1},\cdot)\big(\sum_{j \in J}\xi_j+\eta\big) \prod_{j=1}^{n} \cF G(t_{j+1}-t_j,\cdot)\big(\sum_{s=1}^{j}\xi_s\big) \\
& \qquad \qquad \qquad \prod_{i=1}^{k}1_{\{\xi_{\ell_i}=-\xi_{m_i}=\eta_i\}} \prod_{i=1}^{k}\mu_{\e}(d\eta_i) d{\bf t}.
\end{align*}

Applying H\"older's inequality as above, and using the estimate \eqref{FG-bound-m} for the terms corresponding to $m_i$, we obtain that:
\begin{align*}
& \| G_{n,k,J,I_1,\ldots,I_k}^{\e}(\cdot,t,x)\|_{\cH^{\otimes (n-2k+1)}}^2=\int_{(\bR^d)^{n-2k+1}}|\cF G_{n,k,J,I_1,\ldots,I_k}^{\e}(\cdot,t,x)\big((\xi_j)_{j \in J},\eta \big)|^2 \prod_{j\in J}\mu(d\xi_j)\mu(d\eta)\\
& \quad \leq C_t^k K_{\mu}^k \frac{t^{n+1}}{(n+1)!}\int_{T_{n+1}(t)}\prod_{i=1}^{k}(t_{m_i+1}-t_{m_i})^2 \int_{(\bR^d)^{n-2k+1}} e^{-\e|\eta|^2} \Big(1-e^{-\frac{\e}{2}\sum_{j\in J}|\xi_j|^2} \Big)^2 \\
& \quad \quad \quad \int_{(\bR^d)^k}|\cF G(t-t_{n+1},\cdot)(\sum_{j\in J}\xi_j+\eta)|^2 \prod_{j\in J}|\cF G (t_{j+1}-t_j,\cdot)(\sum_{s=1}^{j}\xi_s)|^2 \\
& \qquad \qquad \quad \prod_{i=1}^{k}|\cF G(t_{\ell_i+1}-t_{\ell_i},\cdot)(\sum_{s=1}^{\ell_i}\xi_s)| \prod_{i=1}^{k}1_{\{\xi_{\ell_i}=-\xi_{m_i}=\eta_i\}} \prod_{i=1}^{k}\mu(d\eta_i) \prod_{j \in J}\mu(d\xi_j) d{\bf t}.
\end{align*}
The last integral converges to $0$ as $\e \downarrow 0$, by the dominated convergence theorem. Hence,
\begin{align*}
& \bE|L_1^{\e}(t,x)| \leq \sum_{n\geq 1}\sum_{k=0}^{\lfloor n/2 \rfloor} \sum_{\substack{J \subset [n]\\ |J|=n-2k}} \sum_{\substack{\{I_1,\ldots,I_k\} \ {\rm partition} \ {\rm of} \ J^c \\ I_i=\{\ell_i,m_i\} \forall i=1,\ldots,k}}\Big(\bE|\cG_{n,k,J,I_1,\ldots,I_k}^{\e}(t,x)|^2\Big)^{1/2}\\
& \quad \leq \sum_{n\geq 1}\sum_{k=0}^{\lfloor n/2 \rfloor} \sum_{\substack{J \subset [n]\\ |J|=n-2k}} \sum_{\substack{\{I_1,\ldots,I_k\} \ {\rm partition} \ {\rm of} \ J^c \\ I_i=\{\ell_i,m_i\} \forall i=1,\ldots,k}}\big((n-2k+1)!\big)^{1/2}\|G_{n,k,J,I_1,\ldots,I_k}^{\e}(t,x)\|_{\cH^{\otimes (n-2k+1)}}.
\end{align*}
The last series converges to $0$ as $\e \downarrow 0$, again by the dominated convergence theorem. The application of this theorem is justified as in the proof of Theorem \ref{th-n-conv}.

\medskip

{\em Step 4.} We treat $L_2^{\e}(t,x)$. It can be proved that the function $G_{n,k,J,I_1,\ldots,I_k}'^{\e}(\cdot,t,x)$ belongs to $L^1((\bR^d)^{n-2k+1})$ and its Fourier transform is
\begin{align*}
& \cF G_{n,k,J,I_1,\ldots,I_k}'^{\e}(\cdot,t,x)\big((\xi_j)_{j \in J},\eta \big)=e^{-i (\sum_{j\in J} \xi_j) \cdot x}e^{-i \eta \cdot x} e^{-\frac{\e}{2}|\eta|^2} \\
& \quad \int_{T_{n+1}(t)} \int_{(\bR^d)^k}\cF G(t-t_{n+1},\cdot)\big(\sum_{j \in J}\xi_j+\eta\big) \prod_{j=1}^{n} \cF G(t_{j+1}-t_j,\cdot)(\sum_{s=1}^{j}\xi_s) \\
& \qquad \qquad \qquad \prod_{i=1}^{k}1_{\{\xi_{\ell_i}=-\xi_{m_i}=\eta_i\}} \Big(1-e^{-\e \sum_{i=1}^{k}|\eta_i|^2}\Big) \prod_{i=1}^{k}\mu(d\eta_i) d{\bf t},
\end{align*}
where ${\bf t}=(t_1,\ldots,t_{n+1})$. As in Step 4, one can show that $\bE|L_2^{\e}(t,x)|\to 0$ as $\e \downarrow 0$.

\medskip

{\em Step 5.} We treat $L_3^{\e}(t,x)$. Expressing the $\cH$-inner product in terms of Fourier transforms, we see that
\begin{align*}
 &R_{n,k,J,I_1,\ldots,I_k}^{\e,\ell}\big((x_j)_{\substack{j \in J \\ j \not=\ell}},t,x\big)=\int_0^t \int_{\bR^d}G(t-s,x-y)\int_{\bR^d} e^{i \xi_{\ell} \cdot y} e^{-\frac{\e}{2}|\xi_{\ell}|^2}  \\
& \qquad \qquad \qquad \left(\int_{\bR^d} e^{-i \xi_{\ell} \cdot x_l} T_{n,k,J,I_1,\ldots,I_k}^{\e}\big( (x_j)_{j\in J},s,y \big) dx_{\ell} \right) \mu(d\xi_{\ell})dyds.
\end{align*}
Therefore,
\begin{align*}
& \cF R_{n,k,J,I_1,\ldots,I_k}^{\e,\ell}(\cdot,t,x)\big( (\xi_j)_{\substack{j \in J \\ j \not=\ell}}\big)=\int_{(\bR^d)^{n-2k-1}} e^{-i \sum_{\substack{j \in J \\ j \not=\ell}}\xi_j \cdot x_j}
R_{n,k,J,I_1,\ldots,I_k}^{\e,\ell}\big( (x_j)_{\substack{j \in J \\ j \not=\ell}},t,x\big)d \big( (x_j)_{\substack{j \in J \\ j \not=\ell}} \big)\\
& \quad =\int_0^t \int_{\bR^d}G(t-s,x-y) \int_{\bR^d} \cF T_{n,k,J,I_1,\ldots,I_k}^{\e}(\cdot,s,y)\big((\xi_j)_{j\in J}\big)e^{i \xi_{\ell}\cdot y} e^{-\frac{\e}{2}|\xi_{\ell}|^2} \mu(d\xi_{\ell})dyds.
\end{align*}
We now insert formula \eqref{def-FTe} for $\cF T_{n,k,J,I_1,\ldots,I_k}^{\e}(\cdot,s,y)$. We use the fact that
\[
\int_{\bR^d}e^{-i \big(\sum_{\substack{j \in J \\ j \not=\ell} } \xi_j\big)\cdot y}G(t-s,x-y)dy=e^{-i \big(\sum_{\substack{j \in J \\ j \not=\ell} } \xi_j\big)\cdot x} \cF G (t-s,\cdot)\big(\sum_{\substack{j \in J \\ j \not=\ell}}\xi_j\big).
\]
Changing the notation $t_{n+1}=s$ and denoting ${\bf t}=(t_1,\ldots,t_{n+1})$, we obtain that
\begin{align*}
& \cF R_{n,k,J,I_1,\ldots,I_k}^{\e,\ell}(\cdot,t,x)\big( (\xi_j)_{\substack{j \in J \\ j \not=\ell}}\big)=e^{-i \big(\sum_{\substack{j \in J \\ j \not=\ell} } \xi_j\big)\cdot x}\int_{T_{n+1}(t)} \int_{(\bR^d)^{k+1}} e^{-\frac{\e}{2}|\xi_{\ell}|^2}\Big(1-e^{-\frac{\e}{2} \sum_{j\in J}|\xi_j|^2}\Big)\\
& \ \cF G(t-t_{n+1},\cdot) \big(\sum_{\substack{j\in J \\ j \not=\ell}}\xi_j \big) \prod_{j=1}^{n}\cF G(t_{j+1}-t_j,\cdot)\big( \sum_{s=1}^{j}\xi_{s} \big) \prod_{i=1}^{k}1_{\{\xi_{\ell_i}=-\xi_{m_i}=\eta_i\}} \prod_{i=1}^{k}\mu_{\e}(d\eta_i)\mu(d\xi_{\ell}) d{\bf t}.
\end{align*}
By applying H\"older's inequality, \eqref{FG-bound-m} and inequality $e^{-x}\leq 1$ for $x\geq 0$, it can be proved that
\begin{align*}
& \|R_{n,k,J,I_1,\ldots,I_k}^{\e,\ell}(\cdot,t,x)\|_{\cH^{\otimes (n-2k-1)}}^2 = \int_{(\bR^d)^{n-2k-1}} |\cF R_{n,k,J,I_1,\ldots,I_k}^{\e,\ell}(\cdot,t,x)\big( (\xi_j)_{\substack{j \in J \\ j \not=\ell}}\big)|^2 \prod_{\substack{j \in J \\ j\not=\ell}}\mu(d\xi_j)\\
& \quad
\leq C_t^{k+1} K_{\mu}^{k+1}\frac{t^{n+1}}{(n+1)!}\int_{T_{n+1}(t)}
\prod_{i=1}^{k}(t_{m_i+1}-t_{m_i})^2 \int_{(\bR^d)^{n-2k-1}}\int_{(\bR^d)^{k+1}}e^{-\e|\xi_{\ell}|^2}\\
& \quad \quad \quad \Big( 1-e^{-\frac{\e}{2}\sum_{j\in J}|\xi_j|^2}\Big)^2
|\cF G(t-t_{n+1},\cdot) \big(\sum_{\substack{j\in J \\ j \not=\ell}}\xi_j \big)|^2 \prod_{j\in J}|\cF G(t_{j+1}-t_j,\cdot)\big( \sum_{s=1}^{j}\xi_{s} \big) |^2 \\
&  \quad \quad \quad  \prod_{i=1}^{k}|\cF G(t_{\ell_i+1}-t_{\ell_i},\cdot)\big( \sum_{s=1}^{\ell_i}\xi_{s} \big) | \prod_{i=1}^{k}1_{\{\xi_{\ell_i}=-\xi_{m_i}=\eta_i\}}
 \prod_{i=1}^{k}\mu(d\eta_i)\mu(d\xi_{\ell})  \prod_{\substack{j \in J \\ j\not=\ell}}\mu(d\xi_j) d{\bf t}.
\end{align*}
The last integral converges to $0$ by the dominated convergence theorem. Hence,
\begin{align*}
& \bE|L_3^{\e}(t,x)| \leq \sum_{n\geq 1}\sum_{k=0}^{\lfloor n/2 \rfloor} \sum_{\substack{J \subset [n]\\ |J|=n-2k}} \sum_{\substack{\{I_1,\ldots,I_k\} \ {\rm partition} \ {\rm of} \ J^c \\ I_i=\{\ell_i,m_i\} \forall i=1,\ldots,k}}\sum_{\ell \in J} \Big(\bE|\cR_{n,k,J,I_1,\ldots,I_k}^{\e,\ell}(t,x)|^2\Big)^{1/2}\\
& \quad \leq \sum_{n\geq 1}\sum_{k=0}^{\lfloor n/2 \rfloor} \sum_{\substack{J \subset [n]\\ |J|=n-2k}} \sum_{\substack{\{I_1,\ldots,I_k\} \ {\rm partition} \ {\rm of} \ J^c \\ I_i=\{\ell_i,m_i\}\forall i=1,\ldots,k}} \sum_{\ell \in J} \big((n-2k-1)!\big)^{1/2}\|R_{n,k,J,I_1,\ldots,I_k}^{\e,\ell}(t,x)\|_{\cH^{\otimes (n-2k+1)}}.
\end{align*}
The last series converges to $0$ as $\e \downarrow 0$, again by the dominated convergence theorem.

\medskip

{\em Step 6.} We treat $L_4^{\e}(t,x)$. Similarly to Step 5, it can be proved that
\begin{align*}
& \cF R_{n,k,J,I_1,\ldots,I_k}'^{\e,\ell}(\cdot,t,x)\big( (\xi_j)_{\substack{j \in J \\ j \not=\ell}}\big)=e^{-i \big(\sum_{\substack{j \in J \\ j \not=\ell} } \xi_j\big)\cdot x} \int_{T_{n+1}(t)} \int_{(\bR^d)^{k+1}} e^{-\frac{\e}{2}|\xi_{\ell}|^2} \Big(1-e^{-\e \sum_{i=1}^{k}|\eta_i|^2}\Big)\\
& \ \cF G(t-t_{n+1},\cdot) \big(\sum_{\substack{j\in J \\ j \not=\ell}}\xi_j \big) \prod_{j=1}^{n}\cF G(t_{j+1}-t_j,\cdot)\big( \sum_{s=1}^{j}\xi_{s} \big) \prod_{i=1}^{k}1_{\{\xi_{\ell_i}=-\xi_{m_i}=\eta_i\}} \prod_{i=1}^{k}\mu(d\eta_i)\mu(d\xi_{\ell}) d{\bf t}
\end{align*}
and $\bE|L_4^{\e}(t,x)| \to 0$ as $\e \downarrow 0$.

This concludes the proof of \eqref{Li-e-zero} and the proof of the theorem.
\end{proof}

\section{Feynman-Kac-type formula}

\label{section-FK}

In this section, we give the proof of Theorem \ref{FK-th}.
First, we introduce the notation:
\begin{align}
\label{def-FK-Hn}
& \bE^{N,X}\Big[1_{\{N_t=n\}}\prod_{i=1}^{n} (\tau_i-\tau_{i-1}) \prod_{i=1}^{n}\dot{W}(X_{\tau_i}^x) \Big]:=\\
\nonumber
& \quad \quad \bE^{N} \Big[ 1_{\{N_t=n\}}  \prod_{i=1}^{n} (\tau_i-\tau_{i-1}) \int_{(\bR^{d})^n}\prod_{i=1}^{n}\frac{G(\tau_i-\tau_{i-1},x_i-x_{i-1})}{\tau_i -\tau_{i-1}}\prod_{i=1}^{n}W^{\circ}(dx_i) \Big],
\end{align}
where $\tau_0=0$, $x_0=x$.
This notation is motivated by the formal relation
$\prod_{i=1}^{n}W^{\circ}(dx_i) =\prod_{i=1}^{n}\dot{W}(x_i)dx_1 \ldots dx_n$ (see
\eqref{formal-V}) and a formal application of Lemma \ref{elem-lem} (similar to the proof of Lemma \ref{lemma-FK-ve}), with $X=(X_{\tau_1}^x,\ldots,X_{\tau_n}^x)$, $Y=N$, $Z=W$,
\[
g(x_1,\ldots,x_n,W)=\prod_{i=1}^{n}\dot{W}(x_i) \quad \mbox{and} \quad
h(N)=1_{\{N_t=n\}} \prod_{i=1}^{n}(\tau_i-\tau_{i-1}),
\]



\medskip

We need an auxiliary result.

\begin{lemma}
For any random variable $Y:\Omega \to \cY$ 
and a non-negative measurable function $f$ on $(\bR^{d})^n \times \cY$, such that $f(\cdot,y) \in \cH^{\otimes n}$ for any $y \in \cY$ and $\bE[f(\cdot,Y)] \in \cH^{\otimes n}$, we have:
\begin{equation}
\label{e-id4}
\int_{\bR^{2d}}\bE[f(x_1,\ldots,y_n,Y)] \prod_{i=1}^{n}W^{\circ}(dx_i)=\bE^Y \left[ \int_{\bR^{2d}}f(x_1,\ldots,x_n,Y) \prod_{i=1}^{n}W^{\circ}(dx_i) \right].
\end{equation}
\end{lemma}

\begin{proof}
We first examine the case $n=1$. Recall that (for deterministic integrands), $W^{\circ}(dx)=W(dx)$.
Let $\mu_Y$ be the law of $Y$. By the stochastic Fubini theorem,
\begin{align}
\nonumber
\bE^{Y}\left[\int_{\bR^d}f(x,Y)W(dx) \right]&=\int_{\cY}\left( \int_{\bR^d}f(x,y)W(dx) \right) \mu_{Y}(dy) \\
\label{e-id3}
& =\int_{\bR^d}\left( \int_{\cY}f(x,y)\mu_{Y}(dy) \right)W(dx)=
\int_{\bR^d}\bE[f(x,Y)] W(dx).
\end{align}
The general case is proved in the same way, using linearity. For instance, for $n=2$, we use the fact that $W^{\circ}(dx_1)W^{\circ}(dx_2)=W(dx_1)W(dx_2)+\gamma(x_1-x_2)dx_1 dx_2$.
\end{proof}

\medskip

The proof of Theorem \ref{FK-th} follows from the following result.

\begin{proposition}
If Assumption A and condition (C) hold, for any $n\geq 1$, $t>0$, $x \in \bR^d$, with probability 1,
\[
H_n(t,x)=\bE^{N,X}\Big[1_{\{N_t=n\}}\prod_{i=1}^{n} (\tau_i-\tau_{i-1}) \prod_{i=1}^{n}\dot{W}(X_{\tau_i}^x) \Big],
\]
where the right-hand-side is defined by \eqref{def-FK-Hn}.
\end{proposition}

\begin{proof}
We treat only the cases $n=1$ and $n=2$. The general case is similar.

 Let $n=1$. Using the Poisson representation \eqref{link-T-N} and relation \eqref{e-id3} (with $Y=N$),
\begin{align*}
H_1(t,x)&=\int_0^t \int_{\bR^d}G(t-t_1,x-x_1) W(dx_1)dt_1=\int_{\bR^d}\left(\int_0^t G(s_1,x_1-x) ds_1\right) W(dx_1)\\
&=e^{t} \int_{\bR^d}\bE\big[G(\tau_1,x_1-x)1_{\{N_t=1\}}\big]W(dx_1)=e^{t} \bE^{N}\left[ \int_{\bR^d} G(\tau_1,x_1-x) 1_{\{N_t=1\}} W(dx_1) \right]\\
&=e^{t} \bE^{N}\left[1_{\{N_t=1\}} \tau_1\int_{\bR^d} \frac{G(\tau_1,x_1-x)}{\tau_1}  W(dx_1) \right]
\end{align*}

If $n=2$, we use the change of variables $s_1=t-t_2,s_2=t-t_1$, $x_1'=x_2$, $x_2'=x_1$, followed by Poisson representation \eqref{link-T-N} and relation \eqref{e-id4} (with $Y=N$):
\begin{align*}
H_2(t,x)&=\int_{T_2(t)}\int_{\bR^{2d}}G(t-t_2,x-x_2)G(t_2-t_1,x_2-x_1)
W^{\circ}(dx_1)W^{\circ}(dx_2)d{\bf t}\\
&=\int_{\bR^{2d}}\left(\int_{T_2(t)}G(s_1,x_1'-x)G(s_2-s_1,x_2'-x_1')d{\bf s}\right) W^{\circ}(dx_2')W^{\circ}(dx_1')\\
&=e^t \int_{\bR^{2d}} \bE \big[ G(\tau_1,x_1-x)G(\tau_2-\tau_1,x_2-x_1)1_{\{N_t=2\}}\big] W^{\circ}(dx_1)W^{\circ}(dx_2)\\
&=e^t  \bE^{N} \left[ \int_{\bR^{2d}}G(\tau_1,x_1-x)G(\tau_2-\tau_1,x_2-x_1)1_{\{N_t=2\}}
W^{\circ}(dx_1)W^{\circ}(dx_2)\right]\\
&=e^t  \bE^{N} \left[ 1_{\{N_t=2\}} \tau_1 (\tau_2-\tau_1) \int_{\bR^{2d}}\frac{G(\tau_1,x_1-x)}{\tau_1}\frac{G(\tau_2-\tau_1,x_2-x_1)}{\tau_2-\tau_1}
W^{\circ}(dx_1)W^{\circ}(dx_2)\right].
\end{align*}

\end{proof}

\section{Smooth noise case}
\label{section-mu-finite}

In this section, we consider the case of a smooth noise $W$ with a finite spectral measure $\mu$. In this case, $\gamma(x)=\int_{\bR^d} e^{-i \xi \cdot x}\mu(d\xi)$ for any $x \in \bR^d$, and we do not need Assumption A.

\medskip

Since $\mu(\bR^d)<\infty$, the sequence $\{p_{\e}(x-\cdot)\}_{\e>0}$ is Cauchy in $\cH$, uniformly in $x \in \bR^d$:
\[
\|p_{\e}(x-\cdot)-p_{\e'}(x-\cdot) \|_{\cH}^2=\int_{\bR^2} \left(e^{-\e|\xi|^2/2}-e^{-\e'|\xi|^2/2} \right)^2 \mu(d\xi) \to 0 \quad \mbox{as $\e,\e' \to 0$}.
\]
We denote its limit by $A(x)$. Then $\lim_{\e \downarrow 0}\|p_{\e}(x-\cdot)-A(x)\|_{\cH} =0$ uniformly in $x \in \bR^d$, and
\[
\|A(x)\|_{\cH}^2=\lim_{\e \downarrow 0} \|p_{\e}(x-\cdot)\|_{\cH}^2=\lim_{\e \downarrow 0}\int_{\bR^d}e^{-\e |\xi|^2}\mu(d\xi)=\mu(\bR^d).
 \]

The random variable
\[
\dot{W}(x):=W(A(x))
\]
is well-defined and $\dot{W}^{\e}(x) \to \dot{W}(x)$ in $L^2(\Omega)$, as $\e \downarrow 0$, uniformly in $x$.
For any $p>0$,
\begin{equation}
\label{bound-We2}
\bE|\dot{W}^{\e}(x)|^{p}=z_{p}(\bE|\dot{W}^{\e}(x)|^{2})^{p/2} \leq z_{p} (\mu(\bR^d))^{p/2}
\end{equation}
\[
\bE|\dot{W}(x)|^{p}=z_{p}(\bE|\dot{W}(x)|^{2})^{p/2} =z_{p} (\mu(\bR^d))^{p/2}
\]

\begin{remark}
{\rm The Gaussian process $\{\dot{W}(x)\}_{x \in \bR^d}$ has stationary increments (in Yaglom sense) and covariance function:
\[
\bE|\dot{W}(x)-\dot{W}(y)|^2=\int_{\bR^d}(1-e^{i\xi \cdot (x-y)})^2 \mu(d\xi).
\]
This process is {\em not the same} as the mean-square derivative process $W'$ of the random field $W(x)=W(1_{[0,x]})$, which is defined as the limit $W'(x)=\lim_{h \to 0}\frac{W(x+h)-W(x)}{h}$ in $L^2(\Omega)$. (This limit exists if $\mu(\bR^d)<\infty$.)
To see this, it suffices to note that $\{W'(x)\}_{x \in \bR^d}$ is a stationary Gaussian process with zero-mean and covariance
\[
\bE[W'(x)W'(y)]=\int_{\bR^d}e^{-i \xi \cdot(x-y)}\mu(d\xi).
\]
}
\end{remark}

\medskip

Let $\overline{v}_n(t,x)=1+\sum_{k=1}^{n}\overline{H}_k(t,x)$, where
\[
\overline{H}_n(t,x)=\int_{T_n(t)} \int_{\bR^{nd}} \prod_{i=1}^{n}G(t_{i+1}-t_i,x_{i+1}-x_i) \prod_{i=1}^{n}\dot{W}(x_i)d{\bf x}d{\bf t}.
\]

Since $\bE|\dot{W}^{\e}(x)|^{p}$ is bounded by a constant that does not depend on $\e$, similarly to \eqref{bound-Hne}, we have:
\[
\|\overline{H}_n(t,x)\|_p \leq \frac{2^{n/2}}{\pi^{1/(2p)}}C_{p}^{n/p} \big(\mu(\bR^d)\big)^{n/2}c_0^n \frac{t^{2n}}{(n!)^{3/2}}.
\]
and the convergence in Lemma \ref{vne-conv} is uniform in $\e$.
Moreover, the proof of Lemma \ref{vne-conv} can be repeated with $\dot{W}^{\e}$ replaced by $\dot{W}$, to infer that $\{\overline{v}_n(t,x)\}_n$ is a Cauchy sequence in $L^p(\Omega)$, uniformly in $(t,x) \in [0,T] \times \bR^d$.

\bigskip

The next lemma deals with the issue of convergence when $\e \downarrow 0$.

\begin{lemma}
If $\mu(\bR^d)<\infty$, then for any $p\geq1$, $T>0$ and $n\geq1$,
\begin{equation}
\label{vne-vn}
\sup_{(t,x) \in [0,T] \times \bR^d}
\bE|H_n^{\e}(t,x)-\overline{H}_{n}(t,x)|^p \to 0 \quad \mbox{as $\e \downarrow 0$}.
\end{equation}
Consequently, for any $t>0,x \in \bR^d$, $H_n(t,x)=\overline{H}_n(t,x)$ a.s. and $v_n(t,x)=\overline{v}_n(t,x)$ a.s.
\end{lemma}

\begin{proof}
By Minkowski's inequality,
\begin{align*}
& \|H_n^{\e}(t,x)-\overline{H}_n(t,x)\|_p \leq \\
& \quad \quad \int_{T_n(t)} \int_{\bR^{nd}} \prod_{i=1}^{n}G(s_{i}-s_{i-1},y_{i}-y_{i-1}) \Big\|\prod_{i=1}^{n}\dot{W}^{\e}(x+y_i)-
\prod_{i=1}^{n}\dot{W}(x+y_i)\Big\|_p d{\bf y}d{\bf s},
\end{align*}
where $s_0=0$ and $y_0=0$.
Using the inequality
\[
\Big|\prod_{i=1}^{n}a_i-\prod_{i=1}^{n}b_i\Big| \leq \sum_{i=1}^{n}|a_1 \ldots a_{i-1}b_{i+1}\ldots b_n| |a_i-b_i|.
\]
for $a_1,b_1,\ldots,a_n,b_n \in \bR$, followed by the generalized H\"older inequality, we obtain:
\begin{align*}
& \bE\Big|\prod_{i=1}^{n}\dot{W}^{\e}(x+y_i)-
\prod_{i=1}^{n}\dot{W}(x+y_i)\Big|^p \leq \\
 & \quad   n^{p-1}\sum_{i=1}^{n} (\bE|\dot{W}^{\e}(x+y_1)|^{pn})^{1/n} \ldots (\bE|\dot{W}^{\e}(x+y_{i-1})|^{pn})^{1/n} (\bE|\dot{W}(x+y_{i+1})|^{pn})^{1/n}\ldots \\
 & \quad \quad \quad (\bE|\dot{W}(x+y_n)|^{pn})^{1/n} (\bE|\dot{W}^{\e}(x+y_i)-\dot{W}(x+y_i)|^{pn})^{1/n} \leq \\
& \quad
n^{p-1} z_{pn} \Big(\mu(\bR^d)\Big)^{\frac{p(n-1)}{2}} \sum_{i=1}^{n}
(\bE|\dot{W}^{\e}(x+y_i)-\dot{W}(x+y_i)|^{2})^{p/2}.
\end{align*}
Since the right-hand-side above converges to $0$ when $\e \downarrow 0$, uniformly in $x \in \bR^d$, the conclusion follows by the dominated convergence theorem.
\end{proof}

\begin{remark}
{\rm If $\mu(\bR^d)<\infty$, it is much easier to prove that $v$ is a Stratonovich solution.
To see this, recall decomposition \eqref{T-decomp}.
Then $\bE|L^{\e}(t,x)|\to 0$ as $\e \downarrow 0$, by the dominated convergence theorem. To justify the application of this theorem, we simply use the bound
\[
\bE|\big(v^{\e}(s,y)-v(s,y)\big)\dot{W}^{\e}(y)| \leq \Big(\bE|v^{\e}(s,y)-v(s,y)|^2\Big)^{1/2} \Big(\bE|\dot{W}^{\e}(y)|^2\Big)^{1/2}.
\]
The last term is bounded by a constant (which does not depend on $\e$) due to \eqref{bound-We2} and the uniform convergence in \eqref{ve-v-conv}.
}
\end{remark}

\begin{remark}
{\rm
If $\mu(\bR^d)<\infty$, $v$ satisfies also the equation:
\[
v(t,x)=1+\int_0^t \int_{\bR^d}G(t-s,x-y)v(s,y)\dot{W}(y)dyds.
\]
To see this, consider the recurrence relation \eqref{rec-vne}. Taking the limit as $\e \downarrow 0$ in $L^1(\Omega)$, we obtain:
\[
v_{n+1}(t,x)=1+\int_0^t \int_{\bR^d}G(t-s,x-y)v_n(s,y)\dot{W}(y)dyds.
\]
Now take the limit as $n \to \infty$ in $L^1(\Omega)$.
}
\end{remark}

\begin{remark}
{\rm If $\mu(\bR^d)<\infty$,
we have the following Feynman-Kac representation of $v$:
\[
v(t,x)=e^{t}\bE^{N,X}\Big[ \prod_{i=1}^{N_t}(\tau_i-\tau_{i-1}) \prod_{i=1}^{N_t}\dot{W}(X_{\tau_i}^x) \Big].
\]
which is proved similarly to \eqref{FK-ve}, replacing $\dot{W}^{\e}$ by $\dot{W}$.
}
\end{remark}

\section{Comparison with the Skorohod solution}
\label{section-comparison}

In this section, we compare the Stratonovich solution $v$ with the Skorohod solution $u$. 

\subsection{Chaos expansions}

In this section, we give the chaos expansions of $v(t,x)$ and $u(t,x)$.

Recall that $u$ satisfies equation \eqref{Skorohod}.
By Theorem 2.2 of \cite{BCC}, we know that under condition (D), $u$ exists and has the chaos expansion $u(t,x)=1+\sum_{n\geq 1}J_n(t,x)$, where $J_n(t,x)=I_n(f_n(\cdot,x;t))$ and $f_n(\cdot,x;t)$ is given by \eqref{def-fn}.
An important observation is that
\[
J_n(t,x)=H_{n,0,[n]}(t,x)=\int_{T_n(t)} \int_{(\bR^d)^n} \prod_{j=1}^{n}G(t_{j+1}-t_j,x_{j+1}-x_j) \prod_{j=1}^{n}W(dx_j)d{\bf t}.
\]
Recalling definition \eqref{def-Hn} of $H_n(t,x)$, we write $H_n(t,x)=J_n(t,x)+M_n(t,x)$, where
\[
M_n(t,x)=
\sum_{k=1}^{\lfloor n/2 \rfloor} \sum_{\substack{J \subset [n]\\ |J|=n-2k}} \sum_{\substack{\{I_1,\ldots,I_k\} \ {\rm partition} \ {\rm of} \ J^c \\ I_i=\{\ell_i,m_i\} \forall i=1,\ldots,k}} H_{n,k,J,I_1,\ldots,I_k}(t,x).
\]
Note that $M_1(t,x)=0$ since $H_1(t,x)=J_1(t,x)$. We have:
\[
v(t,x)=1+\sum_{n\geq 1}J_n(t,x)+\sum_{n\geq 1}M_n(t,x)=u(t,x)+\sum_{n\geq 1}M_n(t,x).
\]
The following table shows shows the composition of $M_n(t,x)$ in comparison with $J_n(t,x)$:
\begin{center}
\begin{tabular}{|c||c||l|c|c|}
\hline
$n$ & $J_n(t,x)$ & $M_n(t,x)$ & $k$ & $n-2k$ \\ \hline 
1 &  $\in \cH_1$ & 0 &  & \\ \hline 
2 &  $\in \cH_2$ & 1 term in $\cH_0$ & 1 & 0 \\ \hline
3 &  $\in \cH_3$ & 3 terms in $\cH_1$ & 1 & 1 \\ \hline
4 &  $\in \cH_4$ & 6 terms in $\cH_2$ & 1 & 2\\
& &  3 terms in $\cH_0$ & 2 & 0 \\ \hline
5 &  $\in \cH_5$ & 10 terms in $\cH_3$ & 1 & 3 \\
& &  15 terms in $\cH_1$ & 2 & 1\\ \hline
6 &  $\in \cH_6$ & 15 terms in $\cH_4$ & 1 & 4\\
& &   45 terms in $\cH_2$ & 2  & 2 \\
& &   15 terms in $\cH_0$ & 3 & 0 \\ \hline
\end{tabular}
\end{center}
In general, $M_n(t,x)$ contains $\binom{n}{2k} \cdot \frac{(2k)!}{2^k k!}$ terms in $\cH_{n-2k}$, for any $k=1,2,\ldots,\lfloor n/2 \rfloor$.

\medskip

The projection of $v(t,x)$ on $\cH_m$ is equal to $J_m(t,x)$ plus the sum of all terms in $\cH_m$ which appear in the 3rd column of the table above.
This shows that
the chaos expansion 
\[
v(t,x)=\sum_{m\geq 0}K_m(t,x) \quad \mbox{with} \quad K_m(t,x) \in \cH_m
\]
of $v(t,x)$ is much more complicated than that of $u(t,x)$. More precisely, \begin{align*}
K_0(t,x)&=1+\sum_{\substack{n\geq 1 \ {\rm even} \\ n/2=:k}} \sum_{\substack{\{I_1,\ldots,I_k\} \ {\rm partition} \ {\rm of} \ [n] \\ I_i=\{\ell_i,m_i\} \forall i=1,\ldots,k}} H_{n,k,\emptyset,I_1,\ldots,I_k}(t,x)=\bE\big(v(t,x)\big) \\
K_m(t,x)&=\sum_{\substack{n\geq 1,\ n-m \ {\rm even} \\ (n-m)/2=:k}} \sum_{\substack{J \subset [n]\\ |J|=m}} \sum_{\substack{\{I_1,\ldots,I_k\} \ {\rm partition} \ {\rm of} \ J^c \\ I_i=\{\ell_i,m_i\} \forall i=1,\ldots,k}} H_{n,k,J,I_1,\ldots,I_k}(t,x) \quad \mbox{for $m\geq 1$}.
\end{align*}
The term of the sum which corresponds to $n=m$ is $H_{m,0,[m]}(t,x)=J_m(t,x)$, for $m\geq 1$.

\subsection{Approximation for the Skorohod solution}

In this section, we present an approximation procedure for $u(t,x)$, which turns out to be quite different than the approximation of $v(t,x)$. This section is included only for the sake of comparison, and it is not needed for the results in this paper.

We consider the equation with mollified noise $\dot{W}^{\e}$ and Wick product $\diamond$:
\begin{align}
\label{approx-equ} 
\begin{cases}
\dfrac{\partial^2 u^{\e}}{\partial t^2} (t,x)=\Delta u^{\e}(t,x)+u^{\e}(t,x) \diamond \dot{W}^{\e}(x), \quad \quad t>0,x \in \bR^d \quad (d \leq 2) \\
u^{\e}(0,x) = 1, \quad \dfrac{\partial u^{\e}}{\partial t}(0,x)=0\end{cases}
\end{align}
Recall that $F \diamond W(h)=\delta(Fh)$ where $\delta$ is the divergence operator and $F$ is a random variable in $L^2(\Omega)$ which is Malliavin differentiable with respect to $W$ (see \cite{nualart06}).

\begin{definition}
\label{Def-ue}
{\rm A process $u^{\varepsilon}=\{u_{\e}(t,x);t\geq 0,x\in \bR^d\}$ is a {\bf solution} to \eqref{approx-equ} if it satisfies
\begin{equation}
\label{ue}
u^{\varepsilon}(t,x)=1+\int_0^t \int_{\bR^d}G(t-s,x-y)u^{\varepsilon}(s,y)\diamond \dot{W}^{\varepsilon}(y)dyds.
\end{equation}
}
\end{definition}

By the definition of the Wick product and stochastic Fubini theorem, \eqref{ue} is equivalent to
\[
u^{\e}(t,x)=1+\int_{\bR^d}\left(\int_{0}^t \int_{\bR^d} G(t-s,x-y) p_{\e}(y-z)u^{\e}(s,y)dydz \right) W(\delta z).
\]
We define the kernel
$$f_{n}^{\e}(y_1,\ldots,y_n,x;t)=\int_{T_n(t)} \int_{(\bR^d)^n}\prod_{i=1}^{n}G(t_{i+1}-t_i,x_{i+1}-x_i) \prod_{i=1}^{n}p_{\e}(x_i-y_i) d{\bf x} d{\bf t}.$$
It is not difficult to prove that $f_{n}^{\e}(\cdot,x;t) \in \cH^{\otimes n}$, using the fact that $f_{n}(\cdot,x;t) \in \cH^{\otimes n}$ and
\[
\cF f_{n}^{\e}(\cdot,x;t)(\xi_1,\ldots,\xi_n)=
\exp\left(-\frac{\varepsilon}{2}\sum_{j=1}^n|\xi_j|^2\right)\cF f_n(\cdot,x;t)(\xi_1,\ldots,\xi_n).
\]

\begin{lemma}
Under condition (D), for any $T>0$ and $p\geq 2$, the series
\begin{equation}
\label{e-series}
u^{\e}(t,x):=1+\sum_{n\geq 1}I_n(f_{n}^{\e}(\cdot,x;t)) \quad \mbox{converges in $L^p(\Omega)$}
\end{equation}
uniformly in $(t,x) \in [0,T] \times \bR^d$ and $\e>0$. Moreover,
$\{u^{\e}(t,x);t\geq 0,x\in \bR^d\}$ is the unique solution of equation \eqref{approx-equ}.
\end{lemma}

\begin{proof} 
Note that
$$\cF \widetilde{f}_{n}^{\e}(\cdot,x;t)(\xi_1,\ldots,\xi_n)=
\exp\left(-\frac{\e}{2}\sum_{j=1}^n|\xi_j|^2\right)\cF \widetilde{f}_n(\cdot,x;t)(\xi_1,\ldots,\xi_n).$$
Hence $\|\widetilde{f}_{n,\varepsilon}(\cdot,x;t)\|_{\cH^{\otimes n}} \leq \|\widetilde{f}_{n}(\cdot,x;t)\|_{\cH^{\otimes n}}$ and the
$L^2(\Omega)$-convergence of the series \eqref{e-series} (uniform in $(t,x)$) follows from the uniform convergence of $\sum_{n\geq 0} n! \|\widetilde{f}_{n}(\cdot,x;t)\|_{\cH^{\otimes n}}^2$. The $L^p(\Omega)$-convergence follows by hypercontractivity.

Next, we prove that $u^{\e}$ is a solution of \eqref{approx-equ}. We have to show that
$u^{\e,t,x} \in {\rm Dom} \, \delta$ and 
$\delta(u^{\e,t,x})=u^{\e}(t,x)-1$, where
$$u^{\e,t,x}(z)=\int_0^t \int_{\bR^d} G(t-s,x-y)p_{\e}(y-z)u^{\e}(s,y)dyds.$$


First, we deduce the chaos expansion of $u^{\e,t,x}(z)$:
\begin{align*}
u^{\e,t,x}(z)&= \int_0^t \int_{\bR^d} G(t-s,x-y) p_{\e}(y-z) \Big(\sum_{n\geq 0} I_n(f_{n}^{\e}(\cdot,y;s)) \Big)dyds\\
&= \sum_{n\geq 0}I_n\left(\int_0^t \int_{\bR^d} G(t-s,x-y) p_{\e}(y-z) f_{n}^{\e}(\cdot,y;s) dy ds\right)\\
&=\sum_{n\geq 0}I_{n}\big(f_{n+1}^{\e}(\cdot,z,x;t)\big).
\end{align*}
Using an analogue of Proposition 1.3.7 of \cite{nualart06},
we infer that
$u^{\e,t,x} \in {\rm Dom}\, \delta$ and $$\delta(u^{\e,t,x})=\sum_{n\geq 0}I_{n+1}\big(f_{n+1}^{\e}(\cdot,x;t)\big)=u^{\e}(t,x)-1.$$
Uniqueness follows by classical arguments.
\end{proof}

\begin{lemma}
\label{u-ue-approx}
Under condition (D), for any $T>0$ and $p\geq 2$,
$$\sup_{(t,x) \in [0,T] \times \bR^d}\bE|u^{\e}(t,x)-u(t,x)|^p \to 0 \quad \mbox{as} \quad \e \downarrow 0.$$
\end{lemma}

\begin{proof}
By hypercontractivity, it suffices to consider the case $p=2$. 
Note that
$u(t,x)-u^{\e}(t,x)=\sum_{n\geq 1} I_n(h_{n}^{\e}(\cdot,x;t))$,
where $h_{n}^{\e}(\cdot,x;t)=f_{n}(\cdot,x;t)-f_{n}^{\e}(\cdot,x;t)$,
and 
$$\cF \widetilde{h}_{n}^{\e}(\cdot,x;t)(\xi_1,\ldots,\xi_n)=\left[1-
\exp\left(-\frac{\e}{2}\sum_{j=1}^n|\xi_j|^2\right)\right]\cF \widetilde{f}_n(\cdot,x;t)(\xi_1,\ldots\xi_n).$$
By the dominated convergence theorem,
$$\|\widetilde{h}_{n}^{\e}(\cdot,x;t)\|_{\cH^{\otimes n}}^2=
\int_{\bR^d}|\cF \widetilde{h}_{n}^{\e}(\cdot,x;t)(\xi_1,\ldots,\xi_n)|^2 \mu(d\xi_1)\ldots \mu(d\xi_n) \to 0 \ \mbox{as} \ \varepsilon \to 0.$$
Moreover,
$\|\widetilde{h}_{n}^{\e}(\cdot,x;t)\|_{\cH^{\otimes n}} \leq \|\widetilde{f}_{n}(\cdot,x;t)\|_{\cH^{\otimes n}}$.
Note that
\begin{align*}
\bE|u(t,x)-u^{\e}(t,x)|^2 & 
=\sum_{n\geq 1} n! \|\widetilde{h}_{n}^{\e}(\cdot,x;t)\|_{\cH^{\otimes n}}^2.
\end{align*}
and the series converges to $0$ as $\e \downarrow 0$, uniformly in $(t,x) \in [0,T] \times \bR^d$, using the dominated convergence theorem and the fact that $\sup_{(t,x)\in [0,T] \times \bR^d}\bE|u(t,x)|^2<\infty$.
\end{proof}

\vspace{3mm}

\appendix

\section{Parseval-type identities}

In this section, we give two Parseval-type identities which are used in this paper.

\begin{lemma}
\label{energy-lemma}
Let $d\geq 1$ be arbitrary and $\gamma:\bR^d \to [0,\infty]$ be non-negative-definite function satisfying Assumption A. Let $\mu$ be the tempered measure on $\bR^d$ such that $\gamma=\cF \mu$.
If $\varphi$ is a non-negative integrable function on $\bR^d$ such that 
$\int_{\bR^{d}}|\cF \varphi(\xi)|^2 \mu(d\xi)<\infty$, then
\[
\int_{\bR^d} \varphi(x)\gamma(x)dx=\int_{\bR^d}\cF \varphi(\xi)\mu(d\xi).
\]
\end{lemma}

\begin{proof} By relation (5.37) of \cite{KX09},
\[
\int_{\bR^d} \int_{\bR^d}\gamma(x-y)\nu_1(dx)\nu_2(dy)=\int_{\bR^d}\cF \nu_1(\xi)\overline{\cF \nu_2(\xi)}\mu(d\xi),
\]
 for any finite measures $\nu_1,\nu_2$ on $\bR^d$ such that 
 $\int_{\bR^d}|\cF \nu_i(\xi)|^2 \mu(d\xi)<\infty$ for $i=1,2$. The conclusion follows taking $\nu_1(dx)=\varphi(x)dx$ and $\nu_2=\delta_0$.
\end{proof}

\begin{lemma}
\label{lemmaF}
Let $\gamma$ and $\mu$ be as in Lemma \ref{energy-lemma}. If $\varphi$ is a non-negative integrable function on $(\bR^d)^{2k}$ and $\{I_1,\ldots,I_k\}$ is a partition of $\{1,\ldots,2k\}$ with $I_i=\{\ell_i,m_i\}$ for $i=1,\ldots,k$, then
\begin{align*}
& \int_{(\bR^d)^{2k}} \varphi(x_1,\ldots,x_{2k}) \prod_{i=1}^{k}\gamma(x_{\ell_i}-x_{m_i})dx_1 \ldots dx_{2k}=\\
& \quad \int_{(\bR^d)^k} \cF \varphi(\xi_1,\ldots,\xi_{2k}) \prod_{i=1}^{k}1_{\{\xi_{\ell_i}=-\xi_{m_i}=\eta_i\}}\mu(d\eta_1) \ldots \mu(d\eta_k),
\end{align*}
provided that the integral on the right-hand-side is finite.
\end{lemma}

\begin{proof} Without loss of generality, we will assume that $I_i=\{2i-1,2i\}$ for all $i=1,\ldots,k$.
(If not, we use the change of variables $y_{1}=x_{\ell_1},y_{2}=x_{m_1},\ldots,y_{2k-1}=x_{\ell_k},y_{2k}=x_{m_k}$ and we let $\rho$ be the permutation of $\{1,\ldots,2k\}$ given by $\ell_1=\rho(1),m_1=\rho(2),\ldots,\ell_k=\rho(2k-1),m_k=\rho(2k)$. Then $y_i=x_{\rho(i)}$ for $i=1,\ldots,2k$, and
\[
\int_{(\bR^d)^{2k}}\varphi(x_1,\ldots,x_{2k})
\prod_{i=1}^{k}\gamma(x_{\ell_i}-x_{m_i}) d{\bf x}=\int_{(\bR^d)^{2k}} \varphi(y_{\rho^{-1}(1)},\ldots y_{\rho^{-1}(2k)}) \prod_{i=1}^{k}\gamma(y_{2i-1}-y_{2i})d{\bf y},
\]
where ${\bf x}=(x_1,\ldots,x_{2k})$ and ${\bf y}=(y_1,\ldots,y_{2k})$.)
We have to prove that
\[
\int_{(\bR^d)^{2k}} \varphi(x_1,\ldots,x_{2k}) \prod_{i=1}^{k}\gamma(x_{2i-1}-x_{2i})d{\bf x}=\int_{(\bR^d)^k} \cF \varphi(\xi_1,\ldots,\xi_{2k}) \prod_{i=1}^{k}1_{\{\xi_{2i-1}=-\xi_{2i}=\eta_i\}} \prod_{i=1}^{k}\mu(d\eta_i),
\]
which we will re-write as:
\begin{equation}
\label{phi-induction}
\int_{(\bR^d)^k}\varphi(x_1,x_1',\ldots,x_k,x_k') \prod_{i=1}^{k}\gamma(x_i-x_i')d{\bf x} d{\bf x'}=\int_{(\bR^d)^k} \cF \varphi(\xi_1,-\xi_1,\ldots,\xi_k,-\xi_k) \prod_{i=1}^{k}\mu(d\xi_i),
\end{equation}
with ${\bf x}=(x_1,\ldots,x_{k})$ and ${\bf x}'=(x_1',\ldots,x_{k}')$.

We prove \eqref{phi-induction} by induction on $k\geq 1$.
If $k=1$, then by Lemma \ref{energy-lemma}, we have:
\begin{align*}
& \int_{(\bR^d)^2}\varphi(x_1,x_1')\gamma(x_1-x_1')dx_1 dx_1' =\int_{\bR^d}
\left( \int_{\bR^d} \varphi(x_1,x_1+y_1) \gamma(y_1)dy_1\right) dx_1\\
& \quad \quad \quad =\int_{\bR^d} \left( \int_{\bR^d} \cF \varphi(x_1,x_1+\cdot)(\xi_1)\mu(d\xi_1)\right) dx_1\\
& \quad \quad \quad =\int_{\bR^d}\int_{\bR^d}\left( \int_{\bR^d} e^{-i \xi_1 \cdot (x_1'-x_1)} \varphi(x_1,x_1') dx_1'\right) \mu(d\xi_1)dx_1\\
&\quad \quad \quad =\int_{\bR^d}\left( \int_{\bR^d}\int_{\bR^d} e^{-i \xi_1 \cdot (x_1-x_1')} \varphi(x_1,x_1') dx_1dx_1' \right)  \mu(d\xi_1)\\
&\quad \quad \quad =\int_{\bR^d} \cF \varphi(\xi_1,-\xi_1)  \mu(d\xi_1),
\end{align*}
where for the second last line we used the fact that $\mu$ is symmetric. Suppose now that the statement holds for $k$. By applying the result for $k=1$, we obtain:
\begin{align*}
&I:=\int_{(\bR^d)^{k+1}}\varphi(x_1,x_1',\ldots,x_{k+1},x_{k+1}')\prod_{i=1}^{k+1}
\gamma(x_i-x_i')dx_1\ldots dx_{k+1}dx_1' \ldots dx_{k+1}'=\\
& \int_{(\bR^d)^{2k}} \left( \int_{\bR^d} \psi_{\xi_{k+1}}(x_1,x_1',\ldots,x_k,x_k') \mu(d\xi_{k+1})\right) \prod_{i=1}^{k}\gamma(x_i-x_i')
dx_1\ldots dx_{k}dx_1' \ldots dx_{k}',
\end{align*}
where $\psi_{\xi_{k+1}}(x_1,x_1',\ldots,x_k,x_k'):=\cF \varphi (x_1,x_1',\ldots,x_k,x_k',\cdot,\cdot)(\xi_{k+1},-\xi_{k+1})$. We apply Fubini's theorem and the induction hypothesis for the inner integral on $(\bR^d)^{2k}$. We obtain:
\[
I=\int_{\bR^d} \left(\int_{(\bR^d)^{k}} \cF \psi_{\xi_{k+1}}(\xi_1,-\xi_1,\ldots,\xi_k,-\xi_k) \mu(d\xi_1)\ldots \mu(d\xi_k)\right) \mu(d\xi_{k+1}).
\]
The conclusion follows since  $\cF \psi_{\xi_{k+1}}(\xi_1,-\xi_1,\ldots,\xi_k,-\xi_k)=\cF \varphi(\xi_1,-\xi_1,\ldots,\xi_{k+1},-\xi_{k+1})$.
\end{proof}

\section{Products of Wiener integrals}

In this section, we develop a formula for the product of $n$ Wiener integrals with respect to the noise $W$, which plays a crucial role in the present paper. This formula may be known but we could not find a reference for it.

We recall the {\em product formula} from Malliavin calculus: for any integer $p\geq 1$, for any symmetric function $f \in \cH^{\otimes p}$ and for any $g \in \cH$,
\begin{equation}
\label{product}
I_p(f)I_1(g)=I_{p+1}(f \otimes g)+pI_{p-1}(f \otimes_1 g),
\end{equation}
where
$f \otimes_{1} g$ is the first contraction of $f$ and $g$, defined by:
\[
(f \otimes_{1} g)(y_1,\ldots,y_{p-1})=\langle f(\cdot,y_1,\ldots,y_{p-1}),g \rangle_{\cH}
\]
(see e.g. Theorem 2.7.10 of \cite{NP12}, or Proposition 1.1.3 of \cite{nualart06}).

\medskip

We will use the fact that $I_p(f)=I_p(\widetilde{f})$ for any $f \in \cH^{\otimes p}$, where $\widetilde{f}$ is the symmetrization of $f$.
We let $f \otimes_0 g=f \otimes g$ and we denote by $f \widetilde{\otimes}_{r} g$ the symmetrization of $f \otimes_{r} g$.

\medskip

\begin{theorem}
\label{product-n}
For any $n\geq 2$ and for any functions $f_1,\ldots,f_n\in \cH$,
\begin{equation}
\label{product-n-eq}
\prod_{j=1}^{n}I_1(f_j)=\sum_{k=0}^{\lfloor n/2 \rfloor} \sum_{\substack{J \subset [n]\\ |J|=n-2k}} \sum_{\substack{\{I_1,\ldots,I_k\} \ {\rm partition} \ {\rm of} \ J^c \\ I_i=\{\ell_i,m_i\} \forall i=1,\ldots,k}}I_{n-2k}(\bigotimes_{j\in J}f_j) \prod_{i=1}^{k}\langle f_{\ell_i},f_{m_i}\rangle_{\cH},
\end{equation}
with the convention that $\otimes_{j \in \emptyset} f_j=1$ and $I_0(1)=1$.
\end{theorem}

\begin{proof}
We use induction on $n$. The case $n=2$ is clear.

Suppose that the result is true for $n$. Assume that $n$ is even, $n=2K$ for $K\in \bZ_{+}$. The case when $n$ is odd is similar and will be omitted.

We multiply relation \eqref{product-n-eq} by $I_{n+1}(f_{n+1})$. To evaluate the product $I_{n-2k}(\bigotimes_{j\in J}f_j)I_1(f_{n+1})$, we dstinguish 2 cases: \\
(a) if $k=K$, then $J=\emptyset$ and $I_{n-2k}(\bigotimes_{j\in J}f_j)I_1(f_{n+1})=I_1(f_{n+1})$;\\
(b) if $k \leq K-1$, then by \eqref{product},
\[
I_{n-2k}(\bigotimes_{j\in J}f_j)I_{1}(f_{n+1})=I_{n+1-2k}\Big(\big(\widetilde{\bigotimes}_{j\in J} f_j \big)\otimes f_{n+1} \Big)+(n-2k)I_{n-2k-1}\Big(\big(\widetilde{\bigotimes}_{j\in J} f_j \big)\otimes_1 f_{n+1} \Big).
\]
Note that $I_{n+1-2k}\Big(\big(\widetilde{\bigotimes}_{j\in J} f_j \big)\otimes f_{n+1} \Big)=I_{n+1-2k}\Big(\bigotimes_{j\in J} f_j \otimes f_{n+1} \Big)$ since
$\big(\widetilde{\bigotimes}_{j\in J} f_j \big)\otimes f_{n+1}$ and $\bigotimes_{j\in J} f_j \otimes f_{n+1}$ have the same symmetrizations.
Moreover, $\big(\widetilde{\bigotimes}_{j \in J} f_j \big) \otimes_1 f_{n+1}=\frac{1}{n-2k} \sum_{j\in J} \langle f_{j},f_{n+1} \rangle_{\cH} \widetilde{\bigotimes}_{\ell \in J \verb2\2 \{j\}} f_{\ell} $ and hence,
\[
I_{n-2k-1}\Big( \big(\widetilde{\bigotimes}_{j\in J}f_j \big) \otimes_1 f_{n+1}\Big)=\frac{1}{n-2k}\sum_{j \in J} \langle f_{j},f_{n+1} \rangle_{\cH} I_{n-2k-1}\Big( \bigotimes_{\ell \in J - \{j\}} f_{\ell} \Big).
\]
It follows that
\begin{align*}
A&:=\prod_{i=1}^{n+1}I_1(f_i)
=\sum_{k=0}^{K-1} \sum_{\substack{J \subset[n]\\ |J|=n-2k}} \sum_{\substack{\{I_1,\ldots,I_k\} \ {\rm partition} \ {\rm of} \ [n]-J\\ I_i=\{\ell_i,m_i\} \forall i=1,\ldots,k}}I_{n+1-2k}\Big(\bigotimes_{j\in J} f_j \otimes f_{n+1} \Big) \prod_{i=1}^{k}\langle f_{\ell_i},f_{m_i}\rangle_{\cH}\\
&+\sum_{k=0}^{K-1} \sum_{\substack{J \subset[n]\\ |J|=n-2k}} \sum_{j \in J} \sum_{\substack{\{I_1,\ldots,I_k\} \ {\rm partition} \ {\rm of} \ [n]-J\\ I_i=\{\ell_i,m_i\} \forall i=1,\ldots,k}}  I_{n-2k-1}\Big( \bigotimes_{\ell \in J - \{j\}} f_{j} \Big)
 \prod_{i=1}^{k}\langle f_{\ell_i},f_{m_i}\rangle_{\cH} \langle f_{j},f_{n+1} \rangle_{\cH} \\
& +\sum_{\substack{(\{I_1,\ldots,I_k\} \ {\rm partition} \ {\rm of} \ [n]-J\\ I_i=\{\ell_i,m_i\} \forall i=1,\ldots,k}} I_1(f_{n+1}) \prod_{i=1}^{K}\langle f_{\ell_i},f_{m_i}\rangle_{\cH}=:A_1+A_2+A_3.
\end{align*}

We have to prove that $A=B$, where
\begin{align*}
B=\sum_{k=0}^{K} \sum_{\substack{J' \subset [n+1]\\ |J'|=n+1-2k}} \sum_{\substack{\{I_1,\ldots,I_k\} \ {\rm partition} \ {\rm of} \ [n+1]-J' \\ I_i=\{\ell_i,m_i\} \forall i=1,\ldots,k}}I_{n+1-2k}(\bigotimes_{j\in J'}f_j) \prod_{i=1}^{k}\langle f_{\ell_i},f_{m_i}\rangle_{\cH}.
\end{align*}

We split $B$ using $k \leq K-1$ and $k=K$: (if $k=K$ then $|J'|=n+1-2K=1$)
\begin{align*}
B&=\sum_{k=0}^{K-1} \sum_{\substack{J' \subset [n+1]\\ |J'|=n+1-2k}} \sum_{\substack{\{I_1,\ldots,I_k\} \ {\rm partition} \ {\rm of} \ [n+1]-J' \\ I_i=\{\ell_i,m_i\} \forall i=1,\ldots,k}}I_{n+1-2k}(\bigotimes_{j\in J'}f_j) \prod_{i=1}^{k}\langle f_{\ell_i},f_{m_i}\rangle_{\cH}\\
&+\sum_{j=1}^{n} \sum_{\substack{\{I_1,\ldots,I_K\} \ {\rm partition} \ {\rm of} \ [n+1]-J' \\ I_i=\{\ell_i,m_i\} \forall i=1,\ldots,K}}I_1(f_j) \prod_{i=1}^{K}\langle f_{\ell_i},f_{m_i}\rangle_{\cH}=:B_1+B_2.
\end{align*}

We split $A_2$ using $k\leq K-2$ and $k=K-1$: (if $k=K-1$ then $n-2k-1=1$)
\begin{align*}
A_2&=\sum_{k=0}^{K-2} \sum_{\substack{J \subset[n]\\ |J|=n-2k}} \sum_{j \in J} \sum_{\substack{\{I_1,\ldots,I_k\} \ {\rm partition} \ {\rm of} \ [n]-J\\ I_i=\{\ell_i,m_i\} \forall i=1,\ldots,k}}  I_{n-2k-1}\Big( \bigotimes_{\ell \in J - \{j\}} f_{j} \Big)
 \prod_{i=1}^{k}\langle f_{\ell_i},f_{m_i}\rangle_{\cH} \langle f_{j},f_{n+1} \rangle_{\cH}\\
 &+\sum_{\substack{J \subset [n] \\ |J|=2}} \sum_{\substack{\{I_1,\ldots,I_{K-1}\} \ {\rm partition} \ {\rm of} \ [n]-J\\ I_i=\{\ell_i,m_i\} \forall i=1,\ldots,K-1}}\sum_{j\in J} I_1\big(\bigotimes_{\ell \in J-\{j\}}f_{\ell}\big)\langle f_j,f_{n+1} \rangle_{\cH} \prod_{i=1}^{K-1} \langle f_{\ell_i}, f_{m_1}\rangle_{\cH}:=A_2'+A_2''.
\end{align*}
We notice that:
\[
(a) \ A_1+A_2'=B_1 \quad \mbox{and} \quad (b) \ A_2''+A_3=B_2.
\]

(a) We use the decomposition $B_1=B_1'+B_1''$, where $B_1'$ and $B_1''$ correspond to the cases $n+1  \in J'$, respectively $n+1 \not \in J'$. We observe that $A_1=B_1''$ and $A_2'=B_1'$.

(b) We use the decomposition $B_2=B_2'+B_2''$, where $B_2'$ and $B_2''$ correspond to the cases $j=n+1$, respectively $j\leq n$. We observe that $A_3=B_2'$ and $A_2''=B_2''$.

\end{proof}

\section{Semigroup-type property of $G$}

In this section, we prove a property of $G$ which was used in Step 3.(a) of the proof of Theorem \ref{th-conv-e}. We denote $G_t(x)=G(t,x)$.

\begin{lemma}[Lemma 4.3 of \cite{BNZ20}]
\label{BNZ-lem4-3}
Suppose that $d=2$. Let $q \in (\frac{1}{2},1)$ be arbitrary. For any $0<r<t$ and $x,z \in \bR^d$
\[
\int_r^t \big(G_{t-s}^{2q}*G_{s-r}^{2q}\big)^{1/q}(x-z) ds \leq A_q (t-r)^{\frac{1}{q}-1}G_{t-r}^{2-\frac{1}{q}}(x-z),
\]
where $A_q>0$ is a constant depending on $q$.
\end{lemma}

As an immediate consequence of the previous lemma, we obtain the following semigroup-type property of $G$.

\begin{lemma}
\label{semigroup}
For any $0<r<t$ and $x,z\in \bR^d$,
\[
\int_r^t \int_{\bR^d} G_{t-s}(x-y)G_{s-r}(y-z)dyds \leq C_t^{(1)} G_{t-r}(x-z),
\]
where $C_t^{(1)} =Ct^2$ and $C$ is a positive constant.
\end{lemma}

\begin{proof} If $d=1$, this follows by relation (2.6) of \cite{nualart-zheng21}.

Suppose that $d=2$. Let $q \in (\frac{1}{2},1)$ be arbitrary. Note that $G_t(x) \leq C_{t,q} G_t^{2q}(x)$, where $C_{t,q}=(2\pi t)^{2q-1}$. By H\"older's inequality and Lemma \ref{BNZ-lem4-3},
\begin{align*}
& \int_r^t\int_{\bR^d}G_{t-s}(x-y)G_{s-r}(y-z)dydr \leq C_{t,q}^2 \int_r^t \big( G_t^{2q}*G_{s}^{2q}\big) (x-z)dr \leq \\
&  \qquad C_{t,q}^2 (t-r)^{1-q} \left[\int_r^t \big( G_t^{2q}*G_{s}^{2q}\big)^{1/q}(x-z) ds\right]^{q} \leq C_{t,q}^2 (t-r)^{2(1-q)} A_q^q G_{t-r}^{2q-1}(x-z).
\end{align*}
The conclusion follows using the fact that $G_{t-r}^{2q-1}(x-z) \leq (2\pi t)^{2-2q}G_{t-r}(x-z)$.
\end{proof}

\noindent \footnotesize{\em Acknowledgement.} The author is grateful to Guangqu Zheng for sharing the proof of Theorem \ref{product-n} in the case $n=3$, and to Samy Tindel for pointing out references \cite{hu-meyer1} and \cite{hu-meyer2}.


\normalsize
\begin{thebibliography}{99}

\bibitem{BCC} Balan, R. M., Chen, L. and Chen, X. (2020).
Exact asymptotics of the stochastic wave equation with time independent noise. Preprint available on arXiv:2007.10203.

\bibitem{balan-song} Balan, R. M. and Song. J. (2017). Hyperbolic Anderson model with space-time homogeneous Gaussian noise. {\em ALEA Latin Am. J. Prob. Math. Stat.} {\bf 14}, 799-849.

\bibitem{BNZ20} Bola\~nos-Guerrero, R., Nualart, D. and Zheng, G. (2020). Averaging 2D stochastic wave equation. Preprint available on arXiv:2003.10346, version 2.

\bibitem{dalang99} Dalang, R. C. (1999). Extending martingale measure stochastic
integral with application to spatially homogenous s.p.d.e.'s. {\em
Electr. J. Probab.} {\bf 4}, paper 6, 1-29.


\bibitem{DMT08} Dalang, R. C., Mueller, C. and Tribe, R. (2008). A
Feynman-Kac-type formula for the deterministic and stochastic wave
equations and other p.d.e.'s. {\em Trans. AMS} {\bf 360}, 4681-4703.

\bibitem{HHNT}
Hu, Y., Huang, J., Nualart, D. and Tindel, S. (2015).
Stochastic heat equations with general multiplicative Gaussian noises: H\"older continuity and intermittency.
{\em Electr. J. Probab.} {\bf 20}, paper no. 55, 50 pp.

\bibitem{hu-meyer1} Hu, Y. and Meyer, P. A. (1988).
Chaos de Wiener et int\'egrales de Feynman. S\'eminaire de Probabilit\'es XXII. Eds. Azema, J., Meyer, P. A. and Yor, M. {\em Lect. Notes Math.} 1321, Springer, 51-71.

\bibitem{hu-meyer2} Hu, Y. and Meyer, P. A. (1988). Sur les int\'egrales multiples de Stratonovitch. S\'eminaire de Probabilit\'es XXII. Eds. Azema, J., Meyer, P. A. and Yor, M. {\em Lect. Notes Math.} 1321, Springer, 72-81.



\bibitem{KX09} Khoshnevisan, D. and Xiao, Y. (2009).
 Harmonic analysis of additive L\'evy processes.
 {\em Probab. Th. Rel. Fields} {\bf 145}, 459-515.



\bibitem{NP12} Nourdin I.  and Peccati G. (2012).  {\em Normal approximations with Malliavin calculus: from Stein's method to universality.} Cambridge Tracts in Mathematics 192. Cambridge University Press, Cambridge.

\bibitem{nualart06} Nualart D. (2006). {\em The Malliavin Calculus and Related Topics}. Second edition. Probability and Its Applications,
Springer-Verlag Berlin Heidelberg.

\bibitem{nualart-zheng21} Nualart, D. and Zheng, G. (2021). Central limit theorems for stochastic wave equations in dimensions one and two. To appear in {\em Stoch. PDEs: Anal. Comp.}
\end{thebibliography}
\end{document}